\documentclass[reqno]{amsart}
\usepackage{amsfonts,amsmath,amsthm,amssymb,mathrsfs}
\usepackage{color}
\usepackage{amsmath,amssymb,amsfonts,bm}
\usepackage{graphicx}
\usepackage{newunicodechar}
\usepackage{xcolor}
\numberwithin{equation}{section}
\usepackage[pdfstartview=FitH,colorlinks=true]{hyperref}
\hypersetup{urlcolor=red, linkcolor=blue,citecolor=red}
\allowdisplaybreaks

\theoremstyle{plain}

\newtheorem{thm}{Theorem}[section]

\newtheorem{prop}[thm]{Proposition}
\newtheorem{cor}[thm]{Corollary}
\newtheorem{lemma}[thm]{Lemma}
\newtheorem{remark}[thm]{Remark}

\begin{document}
\title[Landau equation]
{The analytic Gelfand-Shilov smoothing effect of the Landau equation with hard potential
}

\author[C.-J. Xu \& Y. Xu]{Chao-Jiang Xu and Yan Xu}

\address{Chao-Jiang Xu and Yan Xu
\newline\indent
School of Mathematics and Key Laboratory of Mathematical MIIT,
\newline\indent
Nanjing University of Aeronautics and Astronautics, Nanjing 210016, China
}
\email{xuchaojiang@nuaa.edu.cn; xuyan1@nuaa.edu.cn}

\date{}



\begin{abstract}
In this paper, we study the Cauchy problem of the inhomogeneous Landau equation with hard potentials under the perturbation framework to global equilibrium.
We prove that the solution to the Cauchy problem enjoys the analytic Gelfand-Shilov regularizing effect with a Sobolev initial datum for positive time.
\end{abstract}

\maketitle

\section{Introduction}
The Cauchy problem of spatially inhomogeneous Landau equation reads
\begin{equation}\label{1-1}
\left\{
\begin{aligned}
 &\partial_t F+v\ \cdot\ \partial_x F=Q(F,\ F),\\
 &F|_{t=0}=F_0,
\end{aligned}
\right.
\end{equation}
where $F=F(t, x, v)\ge0$ is the density distribution function at time $t\ge0$ and $x\in\mathbb{T}^3, v\in\mathbb R^3$. The Landau bilinear collision operator is defined by
\begin{equation*}
    Q(G, F)(v)=\sum_{j, k=1}^3\partial_j\bigg(\int_{\mathbb R^3}a_{jk}(v-v_*)[G(v_*)\partial_kF(v)-\partial_kG(v_*)F(v)]dv_*\bigg),
\end{equation*}
where $\left(a_{jk}\right)$ is a symmetric non-negative matrix, with
\begin{equation}\label{matrix A}
   a_{jk}(v)=(\delta_{jk}|v|^2-v_jv_k)|v|^\gamma,\quad \gamma\ge-3.
\end{equation}
Here, $\gamma$ is a parameter which leads to the classification of hard potential if $\gamma>0$, Maxwellian molecules if $\gamma=0$, soft potential if $-3<\gamma<0$ and Coulombian potential if $\gamma=-3$.

The Landau equation is one of the fundamental kinetic equations. In the spatially homogeneous case which does not depend on $x$, there has been extensive work. The smoothness of solutions to the spatially homogeneous Landau equation with hard potential has been investigated by Desvillettes-Villani~\cite{DV}, and to~\cite{CLX1, L-1} for a study of the analytic smoothing, Gevrey regularity was treated in~\cite{CLX2, CLX3}, analytic Gelfand-Shilov smoothing effect was proven in~\cite{LX1}. The situation for Maxwellian molecules, existence, uniqueness and $C^{\infty}(\mathbb R^{3}_{v})$  smoothness for all $t>0$ have been studied in~\cite{V1}, with the initial datum has finite mass and energy. The analytic results can refer to~\cite{MX, M-2}, and the Gelfand-Shilov smoothing effect has been treated in~\cite{LX, M-1}. In the case of soft potentials, the existence and uniqueness results can refer to~\cite{FG, V2, W}. To the smoothing results, ~\cite{LX2} showed that the solution of the linear Landau equation with soft potentials enjoys the analytic smoothness property, and the Gelfand-Shilov smoothing effect for moderately soft potentials was addressed in~\cite{LX3}.

We are now interested in the spatially inhomogeneous case. In this work, we shall study the linearization of the Landau equation \eqref{1-1} near the Maxwellian distribution
\begin{equation*}
    \mu(v)=(2\pi)^{-\frac32}e^{-\frac{|v|^2}{2}}.
\end{equation*}
Considering the fluctuation of the density distribution function
\begin{equation*}
    F(t, x, v)=\mu(v)+\sqrt\mu(v)f(t, x, v),
\end{equation*}
since $Q(\mu, \mu)=0$, the Cauchy problem \eqref{1-1} is reduced to the form
\begin{equation}\label{1-2}
\left\{
\begin{aligned}
 &\partial_tf+v\ \cdot\ \partial_x f+\mathcal Lf=\Gamma(f, f),\\
 &f|_{t=0}=f_0,
\end{aligned}
\right.
\end{equation}
with $F_0=\mu+\sqrt\mu f_0$, where
\begin{equation*}
    \Gamma(f, f)=\mu^{-\frac12}Q(\sqrt\mu f, \sqrt\mu f),
\end{equation*}
\begin{equation*}
\mathcal{L}=\mathcal{L}_1+\mathcal{L}_2,\ \ \ \mbox{with}\ \ \    \mathcal L_1f=-\Gamma(\sqrt\mu, f), \quad \mathcal L_2f=-\Gamma(f, \sqrt\mu).
\end{equation*}

In ~\cite{G-1}, Guo showed the global-in-time existence and uniqueness of solutions to the Landau equation near Maxwellians in regular Sobolev space. In~\cite{CLL}, Chen, Desvillettes and He studied the smoothing effects for classical solutions. $C^{\infty}$ smoothing for the weak solution to the Landau equation with initial data that is bounded by a Gaussian in the velocity variable~\cite{H-1}. The Cauchy problem of the nonlinear equation for Maxwellian molecules, under the perturbation framework to global equilibrium and the $H^{r}_{x}(L^{2}_{v}), r>3/2$ initial datum small enough, enjoys analytic smoothing effect~\cite{M-2}. In \cite{DLSR}, Duan, Liu, Sakamoto and Strain prove the existence of solutions with mind initial data. In~\cite{C-L-X}, Cao, Li and Xu showed that the solution of the Landau equations with hard potential in the perturbation setting enjoys the analytic smoothing effect in both spatial and velocity variables.

Let $\mathcal{A}(\mathbb{T}_x^3 )$ be the analytical space, which is the space of smooth functions satisfying:  $u\in \mathcal{A}(\mathbb{T}_x^3 )$ if there exists a constant $C>0$ such that
$$
\|\partial^{\alpha}_{x}u\|_{L^{2}(\mathbb{T}_x^3)}\le C^{|\alpha|+1}\alpha!, \quad \forall\alpha\in\mathbb N^{3}\, .
$$
For $\sigma, \nu>0$ and $\sigma+\nu\ge 1$, the Gelfand-Shilov space $S^{\sigma}_{\nu}(\mathbb{R}^n)$ is the space of the smooth functions $u$ satisfying there exist the constant $C>0$ such that
$$
\|v^{\beta}\partial^{\alpha}_{v}u\|_{L^{2}(\mathbb R^{n})}\le C^{|\alpha|+|\beta|+1}(\alpha!)^{\sigma}(\beta!)^{\nu},\quad \forall\alpha, \beta\in\mathbb N^{3}.
$$

Now, we define the creation, annulation operators and the gradient of
$$
\mathcal H=-\triangle_v+\frac{|v|^2}{4},
$$
as follows
\begin{equation*}
    A_{\pm, k}=\frac12v_{k}\mp\partial_{v_{k}}, (1\le k\le3), \quad A_{\pm}^{\alpha}=A_{\pm, 1}^{\alpha_{1}}A_{\pm, 2}^{\alpha_{2}}A_{\pm, 3}^{\alpha_{3}}, (\alpha\in\mathbb N^{3}),
\end{equation*}
\begin{equation*}
    \nabla_{\mathcal H_{\pm}}=(A_{\pm, 1}, A_{\pm, 2}, A_{\pm, 3}).
\end{equation*}
For $m\ge1$, we define the norm
\begin{equation}\label{norm}
        \|\nabla_{\mathcal H_{+}}^{m}f\|^{2}_{L^2(\mathbb R^3)}=\sum_{k=1}^{3}\|A_{k, +}\nabla_{\mathcal H_{+}}^{m-1}f\|^{2}_{L^2(\mathbb R^3)}=\sum_{|\alpha|=m}\frac{m!}{\alpha!}\|A_{+}^{\alpha}f\|^{2}_{L^2(\mathbb R^3)}.
\end{equation}
For notation simplicity, we denote $L^{2}_{x, v}=L^{2}(\mathbb T^{3}_{x}\times\mathbb R^{3}_{v})$ and $H^2_xL^2_v=H^2(\mathbb T^3_x; L^2(\mathbb R^3_v))$, for $N, s\ge0$, define the classical Sobolev space
$$H_{x, v}^{N}=\left\{u: \  \sum_{|\alpha|+|\beta|\le N}\|\partial^{\alpha}_{x}\partial^{\beta}_{v}u\|^{2}_{L^{2}_{x, v}}<+\infty\right\},$$
and define the weighted Sobolev space
$$H_{x, v}^{N, s}=\left\{u: \  \sum_{|\alpha|+|\beta|\le N}\|(1+|v|^{2})^{\frac s2}\partial^{\alpha}_{x}\partial^{\beta}_{v}u\|^{2}_{L^{2}_{x, v}}<+\infty\right\}.$$

In this work, we consider the Cauchy problem \eqref{1-2} with $\gamma\ge0$. Combine the results of \cite{C-L-X, CLL, DLSR, G-1}, we suppose that,
there exists a small constant $\epsilon>0$ such that
$$
\|f_{0}\|_{H^2_xL_v^2}\le\epsilon,
$$
 and the Cauchy problem \eqref{1-2} admits a global-in-time solution $f$  satisfying
 \begin{equation}
 	\label{smint}
	\|f\|_{L^{\infty}([0, \infty[; H^2_xL_v^2)} \leq \epsilon,
\end{equation}
with the regularity
\begin{equation}
 	\label{smooth}
C^{\infty}(]0, \infty[; \cap_{s\ge0} H^{\infty, s}_{x, v}(\mathbb T^{3}_{x}\times\mathbb R^{3}_{v})).
\end{equation}

Our main result shows the analytic Gelfand-Shilov regularizing effect of the solution, it can be stated as follows.

\begin{thm}\label{mainresult}
Assume that  $\gamma \geq 0 $, under the assumption \eqref{smint} and \eqref{smooth},
 we have that the solution $f(t)$  of Cauchy problem \eqref{1-2}  belongs to  $ \mathcal{A}(\mathbb{T}_x^3; S^1_1( \mathbb{R}_v^3))$ for all $t >0$.

Moreover, for $\lambda>\frac{3}{2}$, there exists a constant $C>0$ depending on $\gamma$ and $\lambda$,  such that, for all $ \alpha \in \mathbb{N}^3,\  m \in \mathbb{N}, $
\begin{equation}\label{alpha1}
	\tilde{t}^{(\lambda+1)|\alpha|+ \lambda m} \|{\partial_x^{\alpha} \nabla_{\mathcal H_{+}}^{m}  f(t)}\|_{H^2_{x}L^2_v} \leq    C^{|\alpha|+m+1} (|\alpha|+m)!,\quad \forall t>0,
\end{equation}
where $\tilde{t}=\min\{t, 1\}$.
\end{thm}

\begin{remark}
     The above result improve that in~\cite{C-L-X} which shows that  the solution $f(t)$ of Cauchy problem \eqref{1-2} belongs to $\mathcal{A}(\mathbb{T}_x^3\times \mathbb{R}_v^3)$, now we have proven that the solution belongs to $\mathcal{A}(\mathbb{T}_x^3; S^1_1( \mathbb{R}_v^3))$, which means that it is analytic in $\mathbb T^3\times\mathbb R^3$ and exponential decay at infinity with respect to $v$. Otherwise, for $\lambda>\frac{3}{2}$, we have proven that there exist $C_1, C_2$ such that for any $\alpha, \beta, \sigma\in\mathbb N^{3}$,
\begin{equation}\label{v}
\|v^{\beta}\partial^{\alpha}_{v}\partial^\sigma_x f(t)\|_{H^2_{x}L^2_v}\le \left(\frac{C_1}{\tilde{t}^{(\lambda+1)}}\right)^{|\sigma|+1}   \left(\frac{C_2}{\tilde{t}^{\lambda}}\right)^{|\alpha|+|\beta|+1} \alpha!\ \beta!\ \sigma!,\quad \forall t>0\ .
\end{equation}
The above estimates \eqref{v} can be derived from \eqref{alpha1}, by using the Proposition 5.2 of~\cite{LX1}, and Theorem 2.1 of~\cite{GPR}, and the fact that $(m+n)!\le2^{m+n}m!n!$ for all $m, n\in\mathbb N$.

The regularizing effect of  $ f(t)\in \mathcal{A}(\mathbb{T}_x^3; S^1_1( \mathbb{R}_v^3))$ for $t >0$ is sharp for the nonlinear inhomogeneous Landau equation with hard potential $\gamma>0$.
\end{remark}

The paper is arranged as follows. In section \ref{s2}, we give some preliminary results of the Landau operator and the idea of proof for the main Theorem. In section \ref{s3}, we study the commutator of Landau operators with vector fields. In section \ref{s4}, we give the energy estimation for multi-directional derivation, then prove the analytic Gelfand-Shilov regularizing effect of the Landau equation. In section \ref{s5}, we prove some technical Lemmas which had been used in section \ref{s4}.

\section{Preliminary results and idea of proof for main Theorem}\label{s2}

\subsection*{Preliminary results of the Landau Operator}
In the following, the notation $A\lesssim B$ means there exists a constant $C>0$ such that $A\le CB$. For simplicity, with $\gamma\in\mathbb R$, we denote the weighted Lebesgue spaces
\begin{equation*}
    \|\langle\cdot\rangle^\gamma f\|_{L^p(\mathbb R^3)}=\|f\|_{L^p_{\gamma}(\mathbb R^3)},\quad 1\le p\le\infty,
\end{equation*}
where we use the notation $\langle v\rangle=(1+|v|^2)^{\frac12}$. For the
matrix $(a_{jk})$ defined in \eqref{matrix A}, we denote $\bar a_{jk}=a_{jk}*\mu$ and the norm
\begin{equation*}
    |||f|||^2_\sigma=\sum_{j, k=1}^3\int\left(\bar a_{jk}\partial_jf\partial_kf+\frac14\bar a_{jk}v_jv_kf^2\right)dv, \quad f\in\mathcal S(\mathbb R^{3}),
\end{equation*}
and
$$
|||f|||^2_{H^{2}_{x}, \sigma}=\sum_{|\alpha|\le2}\int_{\mathbb T^{3}_{x}}|||\partial^{\alpha}_{x}f(x, \cdot)|||^2_\sigma dx.
$$
From corollary 1 of~\cite{G-1} and Proposition 2.3 of~\cite{LX1}, for $\gamma\ge 0$, one has
\begin{equation}\label{H}
    |||f|||^2_{\sigma}
    \ge C_{1}\|\langle \cdot\rangle^{\frac{\gamma}{2}}\nabla_{\mathcal H_{\pm}} f\|^2_{L^{2}(\mathbb R^{3}_{v})}\ge C_{1}\|\nabla_{\mathcal H_{\pm}} f\|^2_{L^{2}(\mathbb R^{3}_{v})}-\|f\|^{2}_{L^{2}(\mathbb R^{3}_{v})}.
\end{equation}

First, we review the representations for $\mathcal L_{1}$, $\mathcal L_{2}$ and $\Gamma$.
\begin{lemma}~(\cite{LX1}) \label{representations}
For $f, g\in\mathcal S(\mathbb R^{3}_{v})$, we have
$$\mathcal L_{1}f=\sum_{j, k=1}^{3}A_{+, j}\left((a_{jk}*\mu)A_{-, k}f\right),\ \mathcal L_{2}f=-\sum_{j, k=1}^{3}A_{+, j}\left(\sqrt\mu(a_{jk}*(\sqrt\mu A_{-, k}f))\right),$$
$$\Gamma(f, g)=\sum_{j, k=1}^{3}A_{+, j}\left((a_{jk}*(\sqrt\mu f))A_{+, k}g\right)-\sum_{j, k=1}^{3}A_{+, j}\left((a_{jk}*(\sqrt\mu A_{+, k}f))g\right).$$
\end{lemma}
Next, we review the trilinear estimate of the nonlinear Landau operator.
\begin{lemma}~(\cite{LX1})\label{Gamma}
For any $f, g, h\in\mathcal S(\mathbb R^{3}_{v})$ and $\gamma\ge 0$, then
$$\left|(\Gamma(f, g), h)_{L^{2}_{v}}\right|\lesssim\|f\|_{L^{2}}|||g|||_{\sigma}|||h|||_{\sigma}.$$
\end{lemma}
\begin{lemma}
For any $f, g, h\in H^{\infty}(\mathbb T^{3}_{x}; \mathcal S(\mathbb R^{3}_{v}))$ and $\gamma\ge 0$, then
\begin{equation}\label{Gamma-1}
     \left|(\Gamma(f, g), h)_{H^{2}_{x}L^{2}_{v}}\right|\le C_{2}\|f\|_{H^{2}_{x}L^{2}_{v}}|||g|||_{H^{2}_{x}, \sigma}|||h|||_{H^{2}_{x}, \sigma}.
\end{equation}
\end{lemma}
The proof of this proposition is similar to Proposition 4.1 of \cite{M-2}, we need to use the Fourier transformation of $x$ variable to prove $H^{2}_{x}$ is an algebra.

\bigskip
\noindent
{\bf Auxiliary  vector fields and commutators with kinetic operators}\\
We define now a family of auxiliary  vector fields $H_\delta $  which is first introduced by (\cite{CLX})
   \begin{equation*}
   	\label{vecM}
   H_\delta= \frac{1}{\delta+1}t^{\delta+1} \partial_{x_1}- t^{\delta} A_{+, 1},
   \end{equation*}
   where $\delta>\frac{3}{2}$.
The advantage of these vector fields is that the spatial derivatives are not involved in the commutator between $H_\delta$ with the kinetic operators $\partial_t+v \cdot \partial_x $, precisely,  recalling $[\cdot,\, \cdot]$ stands for the commutator between two operators, we  split
\begin{equation*}
	[H_\delta, \  \partial_t+v \cdot \partial_x ]=[H_\delta, \  \partial_t  ]+[H_\delta, \  v \cdot \partial_x ],
\end{equation*}
 and observe the fact that
\begin{align*}
	[H_\delta, \  \partial_t  ]f=H_\delta\partial_t f-\partial_t H_\delta f= -t^{\delta}\partial_{x_1}f +\delta t^{\delta-1} A_{+, 1}f,
\end{align*}
and
\begin{align*}
	[H_\delta, \  v \cdot \partial_x ]f=H_\delta (v \cdot \partial_x f)-(v \cdot \partial_x) H_\delta f=t^{\delta}\partial_{x_1}f.
\end{align*}
This yields
  \begin{equation}\label{keyob}
  	[H_\delta, \,\,  \partial_t+v\,\cdot\,\partial_x ]=\delta t^{\delta-1} A_{+, 1}.
  \end{equation}
  More generally, we have
\begin{equation}\label{kehigher}
\forall\ k\geq 1,\quad 	[H_\delta ^k, \,\,  \partial_t+v\,\cdot\,\partial_x ]=\delta kt^{\delta-1} A_{+, 1} H_\delta^{k-1},
\end{equation}
which can be derived from using induction on $k$. In fact, the validity of \eqref{kehigher} for $k=1$ follows from \eqref{keyob}.  Assume that $k\ge2$ and \eqref{kehigher} holds for all $l\le k-1$, now we would prove it holds true for $l=k$, by using the induction hypothesis, one has
\begin{equation*}
\begin{split}
     &[H_\delta^{k}, \,\,  \partial_t+v\,\cdot\,\partial_x ]\\
     &=H_\delta[H_\delta^{k-1}, \,\,  \partial_t+v\,\cdot\,\partial_x ]+[H_\delta, \,\,  \partial_t+v\,\cdot\,\partial_x ]H_\delta^{k-1}\\
     &=H_{\delta}\left(\delta (k-1)t^{\delta-1} A_{+, 1} H_\delta^{k-2}\right)+\delta t^{\delta-1} A_{+, 1}H_\delta^{k-1}\\
     &=\delta (k-1)t^{\delta-1} A_{+, 1} H_\delta^{k-1}+\delta t^{\delta-1} A_{+, 1}H_\delta^{k-1}=\delta kt^{\delta-1} A_{+, 1} H_\delta^{k-1}.
\end{split}
\end{equation*}

\subsection*{Formation of a basis by auxiliary vector fields}
The second advantage of this family of vector fields is that we can find a basis $H_{\delta_1}, H_{\delta_2}$ which span $\partial_{x_1}, A_{+, 1}$. In fact, take $\lambda>\frac{3}{2}$.   We define   $\delta_1$ and $\delta_2$ in term of $\lambda$ by setting
\begin{equation*}
	\delta_1 =\lambda, \quad 	\delta_2=\frac{\lambda}{2}+\frac{3}{4}.
\end{equation*}
Then
\begin{equation*}
	\delta_1>\delta_2>\frac{3}{2}.
\end{equation*}
With $\delta_1$ and  $\delta_2$ given above, we have
\begin{equation*}
	H_{\delta_1}=\frac{1}{\delta_1+1} t^{\delta_1+1}\partial_{x_1}-t^{\delta_1}A_{+, 1}   ,\quad
  H_{\delta_2}=\frac{1}{\delta_2+ 1} t^{\delta_2+1}\partial_{x_1}-t^{\delta_2}A_{+, 1}.
\end{equation*}
Denote
$$T_{1}=\frac{(\delta_2+ 1)(\delta_1+1)}{\delta_2-\delta_1}  H_{\delta_1} \quad T_{2}=-\frac{(\delta_2+ 1)(\delta_1+1)}{\delta_2-\delta_1} t^{\delta_1-\delta_2}H_{\delta_2},$$
$$T_{3=}\frac{\delta_1+ 1}{\delta_2-\delta_1} H_{\delta_1} \quad T_{4}=-\frac{\delta_2+ 1}{\delta_2-\delta_1}t^{\delta_1-\delta_2}H_{\delta_2},$$
then $\partial_{x_1}$ and $A_{+, 1}$ can be generated by the linear combination of $H_{\delta_1}, H_{\delta_2}$,
\begin{equation}\label{linear combination}
\left\{
\begin{aligned}
&t^{\lambda+ 1}\partial_{x_1}=\frac{(\delta_2+ 1)(\delta_1+1)}{\delta_2-\delta_1}  H_{\delta_1}-\frac{(\delta_2+ 1)(\delta_1+1)}{\delta_2-\delta_1} t^{\delta_1-\delta_2}H_{\delta_2}=T_{1}+T_{2},\\
&t^{\lambda}A_{+, 1}=\frac{\delta_1+ 1}{\delta_2-\delta_1} H_{\delta_1}-\frac{\delta_2+ 1}{\delta_2-\delta_1}t^{\delta_1-\delta_2}H_{\delta_2}=T_{3}+T_{4}.
\end{aligned}
\right.
\end{equation}
This enables to control of the classical derivatives in terms of the directional derivatives in $H_{\delta_1} $ and $H_{\delta_2} $.

\subsection*{Idea of proof for main Theorem \ref{mainresult}}
We have that for all $\alpha_{1}, m\in\mathbb N$,
\begin{equation*}
\begin{split}
     t^{(\lambda+1)\alpha_{1}+\lambda m}\|\partial^{\alpha_{1}}_{x_{1}}A^{m}_{+, 1}f(t)\|_{H^{2}_{x}L^{2}_{v}}=\|(T_{1}+T_{2})^{\alpha_{1}}(T_{3}+T_{4})^{m}f(t)\|_{H^{2}_{x}L^{2}_{v}}.
\end{split}
\end{equation*}
Since $T_j, j=1, \cdots, 4$ are the vector fields with coefficients depend only on $t$, we have that for all $t>0$
\begin{equation}\label{xA}
\begin{split}
     &\|(T_{1}+T_{2})^{\alpha_{1}}(T_{3}+T_{4})^{m}f(t)\|_{H^{2}_{x}L^{2}_{v}}
     \le\left|\frac{(\delta_2+ 1)(\delta_1+1)}{\delta_2-\delta_1}\right|^{\alpha_{1}+m}\\
     &\qquad\quad\times\sum_{j=0}^{\alpha_{1}}\sum_{k=0}^{m}C_{\alpha_{1}}^{j}C^{k}_{m}t^{(\delta_{1}-\delta_{2})(\alpha_{1}+m-j-k)}\|H^{j+k}_{\delta_{1}}H^{\alpha_{1}+m-j-k}_{\delta_{2}}f(t)\|_{H^{2}_{x}L^{2}_{v}}.
\end{split}
\end{equation}
So that to finish the proof of main Theorem \ref{mainresult}, we just need to prove that there exists a constant $A>0$ such that for any $0<t\le T$ and any $m, n\in\mathbb N$,
    \begin{equation}\label{Hmn}
   \|H^{m}_{\delta_{1}}H^{n}_{\delta_{2}}f(t)\|_{H^2_xL^{2}_v}\le A^{m+n-\frac12}(m-2)!(n-2)!.
    \end{equation}
    Plugging \eqref{Hmn} back into \eqref{xA}, then
    \begin{equation*}
\begin{split}
     &\|(T_{1}+T_{2})^{\alpha_{1}}(T_{3}+T_{4})^{m}f(t)\|_{H^{2}_{x}L^{2}_{v}}\\
     &\le\left(2A(T+1)^{\delta_{1}-\delta_{2}}\frac{(\delta_2+ 1)(\delta_1+1)}{\delta_2-\delta_1}\right)^{\alpha_{1}+m}(\alpha_{1}+m)!.
\end{split}
\end{equation*}
Repeat the same computation for the vector fields $H_\delta$ defined with $x_2, A_{+, 2}$ or $x_3, A_{+, 3}$, we can finally obtain \eqref{alpha1} for $0<t\le T=1$, then we repeat the same proof on 
  $t\in [k, k+1]$ for $k\ge 1$ to finish the proof of Theorem \ref{mainresult}. The rest of this paper is to prove \eqref{Hmn}.

\section{Commutators of Landau operators with vector fields}\label{s3}

Now, we consider the commutator's estimate of the Landau operator. By induction, we can obtain the following Leibniz-type formula.
\begin{lemma}\label{lemma3-1}
For $F, G\in H^{\infty}(\mathbb T^{3}_{x}; \mathcal S(\mathbb R^{3}_{v}))$ and $m\ge1$, we have
$$H^{m}_{\delta}\left(a_{jk}*(\sqrt\mu F)G\right)=\sum_{l=0}^{m}C_{m}^{l}\left(a_{jk}*(\sqrt\mu H_{\delta}^{l}F)H_{\delta}^{m-l}G\right).$$
\end{lemma}
\begin{proof}
We would prove it by induction on $m$. For $m=1$, we have
\begin{equation*}
\begin{split}
     H_{\delta}\left(a_{jk}*(\sqrt\mu F)G\right)
     &=\frac{1}{\delta+1}t^{\delta+1}a_{jk}*(\sqrt\mu\partial_{x_{1}} F)G\\
     &+\frac{1}{\delta+1}t^{\delta+1}a_{jk}*(\sqrt\mu F)\partial_{x_{1}} G
     -t^\delta A_{+, 1}\left(a_{jk}*(\sqrt\mu F)G\right)\\
     &=a_{jk}*(\sqrt\mu H_{\delta}F)G+a_{jk}*(\sqrt\mu F)H_{\delta}G,
\end{split}
\end{equation*}
here we use
$$
 A_{+, 1}\left(a_{jk}*(\sqrt\mu F)G\right)=a_{jk}*(\sqrt\mu A_{+, 1}F)G+a_{jk}*(\sqrt\mu F)A_{+, 1}G.
$$
Assume that $m\ge2$ and for all $p\le m-1$
$$H^{p}_{\delta}\left(a_{jk}*(\sqrt\mu F)G\right)=\sum_{l=0}^{p}C_{p}^{l}\left(a_{jk}*(\sqrt\mu H_{\delta}^{l}F)H_{\delta}^{p-l}G\right).$$
Now we consider the case of $p=m$, using the induction hypothesis, it follows that
\begin{equation*}
\begin{split}
     &H_{\delta}^{m}\left(a_{jk}*(\sqrt\mu F)G\right)=H_{\delta}\left(H_{\delta}^{m-1}\left(a_{jk}*(\sqrt\mu F)G\right)\right)\\
     &=\sum_{l=0}^{m-1}C_{m-1}^{l}\left(\left(a_{jk}*(\sqrt\mu H_{\delta}^{l+1}F)H_{\delta}^{m-1-l}G\right)+\left(a_{jk}*(\sqrt\mu H_{\delta}^{l}F)H_{\delta}^{m-l}G\right)\right)\\
     &=\sum_{l=0}^{m}C_{m}^{l}\left(a_{jk}*(\sqrt\mu H_{\delta}^{l}F)H_{\delta}^{m-l}G\right),
\end{split}
\end{equation*}
here we use the fact $C_{m-1}^{l-1}+C_{m-1}^{l}=C_{m}^{l}$.
\end{proof}

Then, we have the following commutator's estimate of the Landau operator.
\begin{prop}\label{H-m-Gamma}
For any $m, n\in\mathbb N$, let $f, g, h\in H^{\infty}(\mathbb T^{3}_{x}; \mathcal S(\mathbb R^{3}_{v}))$ and $\gamma\ge 0$, then for all $\delta_{1}, \delta_{2}>\frac32$
\begin{equation*}
\begin{split}
     &\left|\left(H^{m}_{\delta_{1}}H^{n}_{\delta_{2}}\Gamma(f, g),\ H^{m}_{\delta_{1}}H^{n}_{\delta_{2}}h\right)_{H^{2}_{x}L^{2}_{v}}\right|\\
     &\le C_{2}\sum_{l=0}^{m}\sum_{p=0}^{n}C_{m}^{l}C_{n}^{p}\|H^{l}_{\delta_{1}}H^{p}_{\delta_{2}}f\|_{H^{2}_{x}L^{2}_{v}}|||H^{m-l}_{\delta_{1}}H^{n-p}_{\delta_{2}}g|||_{H^{2}_{x}, \sigma}|||H^{m}_{\delta_{1}}H^{n}_{\delta_{2}}h|||_{H^{2}_{x}, \sigma},
\end{split}
\end{equation*}
where $C_{2}>0$ is independent of $m$ and $n$.
\end{prop}
\begin{proof}
From the representation for $\Gamma$ in Lemma \ref{representations} and Lemma \ref{lemma3-1}, we have
\begin{equation*}
\begin{split}
     &H^{m}_{\delta_{1}}H^{n}_{\delta_{2}}\Gamma(f, g)
     =\sum_{j, k=1}^{3}\sum_{l=0}^{m}\sum_{p=0}^{n}C_{m}^{l}C^{p}_{n}A_{+, j}\left((a_{jk}*(\sqrt\mu H^{l}_{\delta_{1}}H^{p}_{\delta_{2}}f))H^{m-l}_{\delta_{1}}H^{n-p}_{\delta_{2}}A_{+, k}g\right)\\
     &\qquad\quad-\sum_{j, k=1}^{3}\sum_{l=0}^{m}\sum_{p=0}^{n}C_{m}^{l}C_{n}^{p}A_{+, j}\left((a_{jk}*(\sqrt\mu H^{l}_{\delta_{1}}H^{p}_{\delta_{2}}A_{+, k}f))H^{m-l}_{\delta_{1}}H^{n-p}_{\delta_{2}}g\right),
\end{split}
\end{equation*}
noticing that for all $1\le k\le3$ and $j=1, 2$, $[H_{\delta_{j}}, A_{+, k}]=0$, one has
\begin{equation}\label{3-9}
\begin{split}
     H^{m}_{\delta_{1}}H^{n}_{\delta_{2}}\Gamma(f, g)&=\sum_{l=0}^{m}\sum_{p=0}^{n}C_{m}^{l}C_{n}^{p}\Gamma(H^{l}_{\delta_{1}}H^{p}_{\delta_{2}}f, \ H^{m-l}_{\delta_{1}}H^{n-p}_{\delta_{2}}g).
\end{split}
\end{equation}
Then from \eqref{Gamma-1}, we can get that
\begin{equation*}
\begin{split}
     &\left|\left(H^{m}_{\delta_{1}}H^{n}_{\delta_{2}}\Gamma(f, g),\ H^{m}_{\delta_{1}}H^{n}_{\delta_{2}}h\right)_{H^{2}_{x}L^{2}_{v}}\right|\\
     &\le C_{2}\sum_{l=0}^{m}\sum_{p=0}^{n}C_{m}^{l}C_{n}^{p}\|H^{l}_{\delta_{1}}H^{p}_{\delta_{2}}f\|_{H^{2}_{x}L^{2}_{v}}|||H^{m-l}_{\delta_{1}}H^{n-p}_{\delta_{2}}g|||_{H^{2}_{x}, \sigma}|||H^{m}_{\delta_{1}}H^{n}_{\delta_{2}}h|||_{H^{2}_{x}, \sigma}.
\end{split}
\end{equation*}
\end{proof}

Notice that $\mathcal L_{2}f=-\Gamma(f, \sqrt\mu)$, then we have the following result.
\begin{cor}\label{corollary1}
     For any $m, n\in\mathbb N$, let $f, h\in H^{\infty}(\mathbb T^{3}_{x}; \mathcal S(\mathbb R^{3}_{v}))$ and $\gamma\ge 0$, then for all $\delta_{1}, \delta_{2}>\frac32$
\begin{equation*}
\begin{split}
     &\left|\left(H^{m}_{\delta_{1}}H^{n}_{\delta_{2}}\mathcal L_{2}f,\ H^{m}_{\delta_{1}}H^{n}_{\delta_{2}}h\right)_{H^{2}_{x}L^{2}_{v}}\right|
     \le C_{3}\sum_{l=0}^{m}\sum_{p=0}^{n}C_{m}^{l}C^{p}_{n}\left(\sqrt{C_{0}}t^{\delta_{1}}\right)^{m-l}\left(\sqrt{C_{0}}t^{\delta_{2}}\right)^{n-p}\\
     &\qquad\quad\times\sqrt{(m-l+n-p)!}\|H^{l}_{\delta_{1}}H^{p}_{\delta_{2}}f\|_{H^{2}_{x}L^{2}_{v}}|||H^{m}_{\delta_{1}}H^{n}_{\delta_{2}}h|||_{H^{2}_{x}, \sigma},
\end{split}
\end{equation*}
with $C_{0}>0$ and $C_{3}>0$ independent of $m$ and $n$.
\end{cor}
\begin{proof}
     Since $\mathcal L_{2}f=-\Gamma(f, \sqrt\mu)$, from \eqref{3-9}, we have
\begin{equation*}
\begin{split}
     &(H^{m}_{\delta_{1}}H^{n}_{\delta_{2}}\mathcal L_{2}f,\ H^{m}_{\delta_{1}}H^{n}_{\delta_{2}}h)_{H^{2}_{x}L^{2}_{v}}\\
     &=-\sum_{|\alpha|\le2}\sum_{l=0}^{m}\sum_{p=0}^{n}C_{m}^{l}C^{p}_{n}\int_{\mathbb T^{3}_{x}}(\Gamma(H^{l}_{\delta_{1}}H^{n}_{\delta_{2}}\partial^{\alpha}_{x}f, H^{m-l}_{\delta_{1}}H^{n-p}_{\delta_{2}}\sqrt\mu),\ H^{m}_{\delta_{1}}H^{n}_{\delta_{2}}\partial^{\alpha}_{x}h)_{L^{2}_{v}}dx,
\end{split}
\end{equation*}
from Lemma \ref{Gamma}, we have
\begin{equation*}
\begin{split}
     &\left|\int_{\mathbb T^{3}_{x}}(\Gamma(H^{l}_{\delta_{1}}H^{n}_{\delta_{2}}\partial^{\alpha}_{x}f, H^{m-l}_{\delta_{1}}H^{n-p}_{\delta_{2}}\sqrt\mu),\ H^{m}_{\delta_{1}}H^{n}_{\delta_{2}}\partial^{\alpha}_{x}h)_{L^{2}_{v}}dx\right|\\
     &\lesssim\int_{\mathbb T^{3}_{x}}|\|H^{l}_{\delta_{1}}H^{p}_{\delta_{2}}\partial^{\alpha}_{x}f(x, \cdot)\|_{L^{2}_{v}}|||H^{m-l}_{\delta_{1}}H^{n-p}_{\delta_{2}}\sqrt\mu|||_{\sigma}|||H^{m}_{\delta_{1}}H^{n}_{\delta_{2}}\partial^{\alpha}_{x}h(x, \cdot)|||_{\sigma}dx.
\end{split}
\end{equation*}
Notice that
     $$H^{m-l}_{\delta_{1}}H^{n-p}_{\delta_{2}}\sqrt\mu=(-t^{\delta_{1}}A_{+, 1})^{m-l}(-t^{\delta_{2}}A_{+, 1})^{n-p}\sqrt\mu,$$
then we have
\begin{equation*}
\begin{split}
     |||H^{m-l}_{\delta_{1}}H^{n-p}_{\delta_{2}}\sqrt\mu|||^{2}_{\sigma}
     &=t^{2\delta_{1}(m-l)+2\delta_{2}(n-p)}|||A_{+, 1}^{m-l+n-p}\sqrt\mu|||^{2}_{\sigma}\\
     &\le t^{2\delta_{1}(m-l)+2\delta_{2}(n-p)}|||\nabla_{\mathcal H_{+}}^{m-l+n-p}\sqrt\mu|||^{2}_{\sigma},
\end{split}
\end{equation*}
from the Corollary 3.4 of~\cite{LX1}, there exists a constant $C_{0}>0$ such that
\begin{equation*}
\begin{split}
     &|||H^{m-l}_{\delta_{1}}H^{n-p}_{\delta_{2}}\sqrt\mu|||^{2}_{\sigma}\le(C_{0}t^{2\delta_{1}})^{m-l}(C_{0}t^{2\delta_{2}})^{n-p}(m-l+n-p)!.
\end{split}
\end{equation*}
Therefore, we can obtain that
\begin{equation*}
\begin{split}
     &|(H^{m}_{\delta_{1}}H^{n}_{\delta_{2}}\mathcal L_{2}f,\ H^{m}_{\delta_{1}}H^{n}_{\delta_{2}}h)_{H^{2}_{x}L^{2}_{v}}|\lesssim\sum_{|\alpha|\le2}\sum_{l=0}^{m}\sum_{p=0}^{n}C_{m}^{l}C_{n}^{p}(\sqrt{C_{0}}t^{\delta_{1}})^{m-l}(\sqrt{C_{0}}t^{\delta_{2}})^{n-p}\\
     &\quad\times\sqrt{(m-l+n-p)!}\int_{\mathbb T^{3}_{x}}\|H^{l}_{\delta_{1}}H^{p}_{\delta_{2}}\partial^{\alpha}_{x}f(x, \cdot)\|_{L^{2}(\mathbb R^{3})}|||H^{m}_{\delta_{1}}H^{n}_{\delta_{2}}\partial^{\alpha}_{x}h(x, \cdot)|||_{\sigma}dx.
\end{split}
\end{equation*}
By using Cauchy-Schwarz inequality, there exists a constant $C_{3}>0$, independent of $m$ and $n$, such that
\begin{equation*}
\begin{split}
     &|(H^{m}_{\delta_{1}}H^{n}_{\delta_{2}}\mathcal L_{2}f,\ H^{m}_{\delta_{1}}H^{n}_{\delta_{2}}h)_{H^{2}_{x}L^{2}_{v}}|
     \le C_{3}\sum_{l=0}^{m}\sum_{p=0}^{n}C_{m}^{l}C^{p}_{n}\left(\sqrt{C_{0}}t^{\delta_{1}}\right)^{m-l}\left(\sqrt{C_{0}}t^{\delta_{2}}\right)^{n-p}\\
     &\qquad\quad\times\sqrt{(m-l+n-p)!}\|H^{l}_{\delta_{1}}H^{p}_{\delta_{2}}f\|_{H^{2}_{x}L^{2}_{v}}|||H^{m}_{\delta_{1}}H^{n}_{\delta_{2}}h|||_{H^{2}_{x}, \sigma}.
\end{split}
\end{equation*}
\end{proof}

For the operator $\mathcal L_{1}$. We first review that
\begin{lemma}~{\em (\cite{LX1})} 
     For any $s>-3$, we have for $\delta>0$
     \begin{equation}\label{e delta}
          \int_{\mathbb R^{3}}|v-w|^{s}e^{-\delta|w|^{2}}dw\lesssim\langle v\rangle^{s}.
     \end{equation}
\end{lemma}
Then we have the following result.
\begin{prop}\label{H-m-L-1}
For any $m, n\in\mathbb N_{+}$, let $f\in H^{\infty}(\mathbb T^{3}_{x}; \mathcal S(\mathbb R^{3}_{v}))$ and $\gamma\ge 0$, then there exists a constant $C_{4}>0$, independent of $m$ and $n$, such that for all $\delta_{1}, \delta_{2}>\frac32$
\begin{equation*}
\begin{split}
     &\left(H^{m}_{\delta_{1}}H^{n}_{\delta_{2}}\mathcal L_{1}f,\ H^{m}_{\delta_{1}}H^{n}_{\delta_{2}}f\right)_{H^{2}_{x}L^{2}_{v}}\ge\frac34|||H^{m}_{\delta_{1}}H^{n}_{\delta_{2}}f|||_{H^{2}_{x}, \sigma}^{2}-C_{4}\|\langle v\rangle^{\frac\gamma2}H^{m}_{\delta_{1}}H^{n}_{\delta_{2}}f\|^{2}_{H^{2}_{x}L^{2}_{v}}\\
     &-C_{4}\sum_{l=1}^{m}C^{l}_{m}l^{\frac{[\gamma]}{2}+2}\sqrt{l!}t^{\delta_{1} l}|||H^{m-l}_{\delta_{1}}H^{n}_{\delta_{2}}f|||_{H^{2}_{x}, \sigma}|||H^{m}_{\delta_{1}}H^{n}_{\delta_{2}}f|||_{H^{2}_{x}, \sigma}\\
     &-C_{4}\sum_{l=0}^{m}C_{m}^{l}l^{\frac{[\gamma]}{2}+1}\sqrt{l!}nt^{\delta_{1}l}t^{\delta_{2}}|||H^{m-l}_{\delta_{1}}H^{n-1}_{\delta_{2}}f|||_{H^{2}_{x}, \sigma}|||H^{m}_{\delta_{1}}H^{n}_{\delta_{2}}f|||_{H^{2}_{x}, \sigma}\\
     &-C_{4}\sum_{p=1}^{n}C^{p}_{n}p^{\frac{[\gamma]}{2}+2}\sqrt{p!}t^{\delta_{2} p}|||H^{m}_{\delta_{1}}H^{n-p}_{\delta_{2}}f|||_{H^{2}_{x}, \sigma}|||H^{m}_{\delta_{1}}H^{n}_{\delta_{2}}f|||_{H^{2}_{x}, \sigma}\\
     &-C_{4}\sum_{p=0}^{n}C_{n}^{p}p^{\frac{[\gamma]}{2}+1}\sqrt{p!}mt^{\delta_{1}}t^{\delta_{2}p}|||H^{m-1}_{\delta_{1}}H^{n-p}_{\delta_{2}}f|||_{H^{2}_{x}, \sigma}|||H^{m}_{\delta_{1}}H^{n}_{\delta_{2}}f|||_{H^{2}_{x}, \sigma}\\
     &-C_{4}\sum_{l=1}^{m}\sum_{p=1}^{n}C_{n}^{p}C_{m}^{l}(p+l)^{\frac{[\gamma]}{2}+2}\sqrt{(p+l)!}t^{\delta_{1} l}t^{\delta_{2} p}|||H^{m-l}_{\delta_{1}}H^{n-p}_{\delta_{2}}f|||_{H^{2}_{x}, \sigma}|||H^{m}_{\delta_{1}}H^{n}_{\delta_{2}}f|||_{H^{2}_{x}, \sigma}.
\end{split}
\end{equation*}
\end{prop}
\begin{proof}
     From the representation for $\mathcal L_{1}$ in Lemma \ref{representations}, it follows that
\begin{equation*}
\begin{split}
     &\left(H^{m}_{\delta_{1}}H^{n}_{\delta_{2}}\mathcal L_{1}f,\ H^{m}_{\delta_{1}}H^{n}_{\delta_{2}}f\right)_{H^{2}_{x}L^{2}_{v}}=\sum_{|\alpha|\le2}(\partial^{\alpha}_{x}H^{m}_{\delta_{1}}H^{n}_{\delta_{2}}\mathcal L_{1}f,\ \partial^{\alpha}_{x}H^{m}_{\delta_{1}}H^{n}_{\delta_{2}}f)_{L^{2}_{x, v}}\\
     &=\sum_{|\alpha|\le2}\sum_{j, k=1}^{3}(H^{m}_{\delta_{1}}H^{n}_{\delta_{2}}A_{+, j}\left((a_{jk}*\mu)A_{-, k}\partial^{\alpha}_{x}f\right),\ H^{m}_{\delta_{1}}H^{n}_{\delta_{2}}\partial^{\alpha}_{x}f)_{L^{2}_{x, v}},
\end{split}
\end{equation*}
noting that for all $1\le j\le3$ and $l=1, 2$, $[H_{\delta_{l}}, A_{+, j}]=0$, then from integration by parts and Lemma \ref{lemma3-1}, we have
\begin{equation*}
\begin{split}
     &(H^{m}_{\delta_{1}}H^{n}_{\delta_{2}}\mathcal L_{1}f,\ H^{m}_{\delta_{1}}H^{n}_{\delta_{2}}f)_{H^{2}_{x}L^{2}_{v}}=\sum_{|\alpha|\le2}\sum_{j, k=1}^{3}\sum_{l=0}^{m}\sum_{p=0}^{n}C_{m}^{l}C_{n}^{p}\\
     &\quad\times(a_{jk}*(\sqrt\mu H^{l}_{\delta_{1}}H^{p}_{\delta_{2}}\sqrt\mu)H^{m-l}_{\delta_{1}}H^{n-p}_{\delta_{2}}A_{-, k}\partial^{\alpha}_{x}f,\ A_{-, j}H^{m}_{\delta_{1}}H^{n}_{\delta_{2}}\partial^{\alpha}_{x}f)_{L^{2}_{x, v}}.
\end{split}
\end{equation*}
Next, we would show that
\begin{equation}\label{diff}
     \sqrt\mu H^{l}_{\delta_{1}}H^{p}_{\delta_{2}}\sqrt\mu=t^{\delta_{1}l}t^{\delta_{2}p}\partial^{l+p}_{v_{1}}\mu, \ \forall \ l+p\ge1.
\end{equation}
We first consider $p=0$, by direct calculation, one gets $H_{\delta_{1}}\sqrt\mu=t^{\delta_{1}}\sqrt\mu^{-1}\partial_{v_{1}}\mu$, assume it holds for $l-1$, then
\begin{equation*}
\begin{split}
     &\sqrt\mu H^{l}_{\delta_{1}}\sqrt\mu=\sqrt\mu H_{\delta_{1}}H^{l-1}_{\delta_{1}}\sqrt\mu=t^{\delta_{1}(l-1)}\sqrt\mu H_{\delta_{1}}(\sqrt\mu^{-1}\partial^{l-1}_{v_{1}}\mu)=t^{\delta_{1}l}\partial^{l}_{v_{1}}\mu,
\end{split}
\end{equation*}
here we use the fact
$$\sqrt\mu^{-1}\partial_{v_{j}}F=\left(\partial_{v_{j}}-\frac{v_{j}}{2}\right)\left(\sqrt\mu^{-1}F\right)=-A_{+, j}\left(\sqrt\mu^{-1}F\right),\quad  \forall F\in\mathcal S(\mathbb R^{3}_{v}).$$
Since $H_{\delta_{1}}H_{\delta_{2}}=H_{\delta_{2}}H_{\delta_{1}}$,
it follows that
\begin{equation*}
\begin{split}
     \sqrt\mu H^{l}_{\delta_{1}}H_{\delta_{2}}\sqrt\mu&=\sqrt\mu H_{\delta_{2}}H^{l}_{\delta_{1}}\sqrt\mu=t^{\delta_{1}l}\sqrt\mu H_{\delta_{2}}\left(\sqrt\mu^{-1}\partial^{l}_{v_{1}}\mu\right)=t^{\delta_{1}l}t^{\delta_{2}}\partial^{l+1}_{v_{1}}\mu,
\end{split}
\end{equation*}
then we can obtain that \eqref{diff} holds. So that
\begin{equation*}
\begin{split}
     &\left(H^{m}_{\delta_{1}}H^{n}_{\delta_{2}}\mathcal L_{1}f,\ H^{m}_{\delta_{1}}H^{n}_{\delta_{2}}f\right)_{H^{2}_{x}L^{2}_{v}}\\
     &=\sum_{|\alpha|\le2}\sum_{j, k=1}^{3}\left(\bar a_{jk}A_{-, k}H^{m}_{\delta_{1}}H^{n}_{\delta_{2}}\partial^{\alpha}_{x}f,\ A_{-, j}H^{m}_{\delta_{1}}H^{n}_{\delta_{2}}\partial^{\alpha}_{x}f\right)_{L^{2}_{x, v}}\\
     &\ +\sum_{|\alpha|\le2}\sum_{j, k=1}^{3}\left(\bar a_{jk}\left[H^{m}_{\delta_{1}}H^{n}_{\delta_{2}}, A_{-, k}\right]\partial^{\alpha}_{x}f,\ A_{-, j}H^{m}_{\delta_{1}}H^{n}_{\delta_{2}}\partial^{\alpha}_{x}f\right)_{L^{2}_{x, v}}\\
     &\ +\sum_{|\alpha|\le2}\sum_{j, k=1}^{3}\sum_{l=1}^{m}C_{m}^{l}t^{\delta_{1}l}\left(\partial^{l}_{v_{1}}\bar a_{jk}H^{m-l}_{\delta_{1}}H^{n}_{\delta_{2}}A_{-, k}\partial^{\alpha}_{x}f,\ A_{-, j}H^{m}_{\delta_{1}}H^{n}_{\delta_{2}}\partial^{\alpha}_{x}f\right)_{L^{2}_{x, v}}\\
     &\ +\sum_{|\alpha|\le2}\sum_{j, k=1}^{3}\sum_{p=1}^{n}C_{n}^{p}t^{\delta_{2}p}\left(\partial^{p}_{v_{1}}\bar a_{jk}H^{m}_{\delta_{1}}H^{n-p}_{\delta_{2}}A_{-, k}\partial^{\alpha}_{x}f,\ A_{-, j}H^{m}_{\delta_{1}}H^{n}_{\delta_{2}}\partial^{\alpha}_{x}f\right)_{L^{2}_{x, v}}\\
     &\ +\sum_{|\alpha|\le2}\sum_{j, k=1}^{3}\sum_{l=1}^{m}\sum_{p=1}^{n}C_{n}^{p}C_{m}^{l}t^{\delta_{1}l}t^{\delta_{2}p}\left(\partial^{p+l}_{v_{1}}\bar a_{jk}H^{m-l}_{\delta_{1}}H^{n-p}_{\delta_{2}}A_{-, k}\partial^{\alpha}_{x}f,\ A_{-, j}H^{m}_{\delta_{1}}H^{n}_{\delta_{2}}\partial^{\alpha}_{x}f\right)_{L^{2}_{x, v}}\\
     &=I_{1}+I_{2}+I_{3}+I_{4}+I_{5},
\end{split}
\end{equation*}
here $\bar a_{jk}=a_{jk}*\mu$. For the term $I_{1}$, since
$$\sum_{j=1}^{3}a_{jk}(v)v_{j}=\sum_{k=1}^{3}a_{jk}(v)v_{k}=0,$$
then we have
\begin{equation*}
\begin{split}
     I_{1}&=|||H^{m}_{\delta_{1}}H^{n}_{\delta_{2}}f|||_{H^{2}_{x}, \sigma}^{2}-\sum_{|\alpha|\le2}\sum_{j=1}^{3}\int_{\mathbb T^{3}_{x}\times\mathbb R^{3}_{v}}\phi_{j}H^{m}_{\delta_{1}}H^{n}_{\delta_{2}}\partial^{\alpha}_{x}f(x, v)\partial_{v_{j}}H^{m}_{\delta_{1}}H^{n}_{\delta_{2}}\partial^{\alpha}_{x}f(x, v)dxdv\\
     &=|||H^{m}_{\delta_{1}}H^{n}_{\delta_{2}}f|||_{H^{2}_{x}, \sigma}^{2}+\tilde I_{1},
\end{split}
\end{equation*}
with
$$\phi_{j}=\sum_{k=1}^{3}a_{jk}*(v_{k}\mu).$$
From integration by parts, one has
\begin{equation*}
\begin{split}
     \tilde I_{1}=\frac12\sum_{|\alpha|\le2}\sum_{j=1}^{3}\int_{\mathbb T^{3}_{x}\times\mathbb R^{3}_{v}}\partial_{v_{j}}\phi_{j}|H^{m}_{\delta_{1}}H^{n}_{\delta_{2}}\partial^{\alpha}_{x}f(x, v)|^{2}dxdv,
\end{split}
\end{equation*}
by using \eqref{e delta}, we have
$$|\partial_{v_{j}}\phi_{j}(v)|\lesssim\langle v\rangle^{\gamma+1},$$
then from Cauchy-Schwarz inequality and corollary 1 of~\cite{G-1}, we can get that
\begin{equation*}
\begin{split}
     \left|\tilde I_{1}\right|&\lesssim\sum_{|\alpha|\le2}\int_{\mathbb T^{3}_{x}}\|\langle\cdot\rangle^{\frac\gamma2}H^{m}_{\delta_{1}}H^{n}_{\delta_{2}}\partial^{\alpha}_{x}f(x, \cdot)\|_{L^{2}_{v}}\|\langle \cdot\rangle^{\frac\gamma2+1}H^{m}_{\delta_{1}}H^{n}_{\delta_{2}}\partial^{\alpha}_{x}f(x, \cdot)\|_{L^{2}_{v}}dx\\
     &\lesssim\|\langle v\rangle^{\frac\gamma2}H^{m}_{\delta_{1}}H^{n}_{\delta_{2}}f\|_{H^{2}_{x}L^{2}_{v}}|||H^{m}_{\delta_{1}}H^{n}_{\delta_{2}}f|||_{H^{2}_{x}, \sigma}.
\end{split}
\end{equation*}
For the term $I_{2}$, since
$$[H_{\delta_{j}}, A_{-, k}]=-t^{\delta}[A_{+, 1}, A_{-, k}]=0, \ (k\ne1), \ j=1, 2,$$
and
$$[H_{\delta_{j}}, A_{-, 1}]=-t^{\delta}[A_{+, 1}, A_{-, 1}]=t^{\delta_{j}}, \ j=1, 2,$$
one can deduce that
\begin{equation*}
\begin{split}
     H^{m}_{\delta_{1}}H^{n}_{\delta_{2}}A_{-, k}=A_{-, k}H^{m}_{\delta_{1}}H^{n}_{\delta_{2}}, (k\ne1),
\end{split}
\end{equation*}
and
$$H^{m}_{\delta_{1}}H^{n}_{\delta_{2}}A_{-, 1}=A_{-, 1}H^{m}_{\delta_{1}}H^{n}_{\delta_{2}}+mt^{\delta_{1}}H^{m-1}_{\delta_{1}}H^{n}_{\delta_{2}}+nt^{\delta_{2}}H^{m}_{\delta_{1}}H^{n-1}_{\delta_{2}},$$
these lead to
\begin{equation*}
\begin{split}
     I_{2}&=mt^{\delta_{1}}\sum_{|\alpha|\le2}\sum_{j=1}^{3}\int_{\mathbb T^{3}_{x}\times\mathbb R^{3}_{v}}\bar a_{j1}H^{m-1}_{\delta_{1}}H^{n}_{\delta_{2}}\partial^{\alpha}_{x}f(x, v)A_{-, j}H^{m}_{\delta_{1}}H^{n}_{\delta_{2}}\partial^{\alpha}_{x}f(x, v)dxdv\\
     &\quad+nt^{\delta_{2}}\sum_{|\alpha|\le2}\sum_{j=1}^{3}\int_{\mathbb T^{3}_{x}\times\mathbb R^{3}_{v}}\bar a_{j1}H^{m}_{\delta_{1}}H^{n-1}_{\delta_{2}}\partial^{\alpha}_{x}f(x, v)A_{-, j}H^{m}_{\delta_{1}}H^{n}_{\delta_{2}}\partial^{\alpha}_{x}f(x, v)dxdv\\
     &=I_{2, 1}+I_{2, 2}.
\end{split}
\end{equation*}
Decomposing $\mathbb R^{3}\times\mathbb R^{3}$ as
$$\Omega_{1}\cup\Omega_{2}\cup\Omega_{3}=\{|v|\le1\}\cup\{|v|\ge1, 2|w|\ge|v|\}\cup\{|v|\ge1, 2|w|\le|v|\},$$
in the region $\Omega_{1}$, from \eqref{e delta}, we have $|\bar a_{jk}(v)|\le\langle v\rangle^{\gamma+1}$, in the region $\Omega_{2}$, since $|v-w|\le |v|+|w|\le 3|w|$ and $\langle\cdot\rangle^{r}\mu^{s}\in L^{\infty}$ for all $r, s\ge0$, then from \eqref{e delta}, one has
\begin{equation*}
\begin{split}
     \left|\int_{2|w|\ge |v|}a_{j1}(v-w)\mu(w)dw\right|\lesssim\int_{\mathbb R^{3}}|v-w|^{\gamma+1}\mu^{\frac12}(w)dw\lesssim\langle v\rangle^{\gamma+1},
\end{split}
\end{equation*}
so that by using the Cauchy-Schwarz inequality and \eqref{H}, we have
\begin{equation*}
\begin{split}
     &\left|\sum_{|\alpha|\le2}\sum_{j=1}^{3}\int_{\mathbb T^{3}_{x}\times\Omega_{p}}a_{j1}(v-w)\mu(w)H^{m-1}_{\delta_{1}}H^{n}_{\delta_{2}}f(x, v)A_{-, j}H^{m}_{\delta_{1}}H^{n}_{\delta_{2}}f(x, v)dwdvdx\right|\\
     &\lesssim\sum_{|\alpha|\le2}\sum_{j=1}^{3}\int_{\mathbb T^{3}_{x}}\left\|\langle \cdot\rangle^{\frac\gamma2+1}H^{m-1}_{\delta_{1}}H^{n}_{\delta_{2}}\partial^{\alpha}_{x}f\right\|_{L^{2}_{v}}\left\|\langle\cdot\rangle^{\frac\gamma2}A_{-, j}H^{m}_{\delta_{1}}H^{n}_{\delta_{2}}f(x, \cdot)\right\|_{L^{2}_{v}}dx\\
     &\lesssim|||H^{m-1}_{\delta_{1}}H^{n}_{\delta_{2}}f|||_{H^{2}_{x}, \sigma}|||H^{m}_{\delta_{1}}H^{n}_{\delta_{2}}f|||_{H^{2}_{x}, \sigma},
\end{split}
\end{equation*}
here $p=1, 2$. In the region $\Omega_{3}$, expanding $a_{jk}(v-w)$ to get
$$a_{jk}(v-w)=a_{jk}(w)+\sum_{l=1}^{3}\int_{0}^{1}\partial_{l}a_{jk}(v-sw)dsw_{l},$$
we can obtain that
\begin{equation*}
\begin{split}
     &\int_{\Omega_{3}}a_{j1}(v-w)\mu(w)H^{m-1}_{\delta_{1}}H^{n}_{\delta_{2}}\partial^{\alpha}_{x}f(x, v)A_{-, j}H^{m}_{\delta_{1}}H^{n}_{\delta_{2}}\partial^{\alpha}_{x}f(x, v)dwdv\\
     &=\int_{\Omega_{3}}a_{j1}(w)\mu(w)H^{m-1}_{\delta_{1}}H^{n}_{\delta_{2}}\partial^{\alpha}_{x}f(x, v)A_{-, j}H^{m}_{\delta_{1}}H^{n}_{\delta_{2}}\partial^{\alpha}_{x}f(x, v)dwdv\\
     &\ +\sum_{l=1}^{3}\int_{0}^{1}\int_{\Omega_{3}}\partial_{l}a_{jk}(v-sw)\mu(w)w_{l}H^{m-1}_{\delta_{1}}H^{n}_{\delta_{2}}\partial^{\alpha}_{x}f(x, v) A_{-, j}H^{m}_{\delta_{1}}H^{n}_{\delta_{2}}\partial^{\alpha}_{x}f(x, v)dwdvds.
\end{split}
\end{equation*}
Since $\gamma\ge0$, it follow that
$$\left|\partial_{l}a_{jk}(v-sw)\right|\lesssim |v-sw|^{\gamma+1}\lesssim|v|^{\gamma+1}+|w|^{\gamma+1}\lesssim\langle v\rangle^{\gamma+1}, \quad \forall \ 0<s<1,$$
notice that $\langle\cdot\rangle^{r}\mu\in L^{\infty}(\mathbb R^{3}), \forall r\ge0$, then from Cauchy-Schwarz inequality, one has
\begin{equation*}
\begin{split}
     &\left|\int_{\Omega_{3}}a_{j1}(v-w)\mu(w)H^{m-1}_{\delta_{1}}H^{n}_{\delta_{2}}\partial^{\alpha}_{x}f(x, v)A_{-, j}H^{m}_{\delta_{1}}H^{n}_{\delta_{2}}\partial^{\alpha}_{x}f(x, v)dwdv\right|\\
     &\lesssim\left\|H^{m-1}_{\delta_{1}}H^{n}_{\delta_{2}}\partial^{\alpha}_{x}f(x, \cdot)\right\|_{L^{2}(\mathbb R^{3}_{v})}\left\|A_{-, j}H^{m}_{\delta_{1}}H^{n}_{\delta_{2}}\partial^{\alpha}_{x}f(x, \cdot)\right\|_{L^{2}(\mathbb R^{3}_{v})}\\
      &\quad+\left\|\langle\cdot\rangle^{\frac\gamma2+1}H^{m-1}_{\delta_{1}}H^{n}_{\delta_{2}}\partial^{\alpha}_{x}f(x, \cdot)\right\|_{L^{2}(\mathbb R^{3}_{v})}\left\|\langle\cdot\rangle^{\frac\gamma2}A_{-, j}H^{m}_{\delta_{1}}H^{n}_{\delta_{2}}\partial^{\alpha}_{x}f(x, \cdot)\right\|_{L^{2}(\mathbb R^{3}_{v})},
\end{split}
\end{equation*}
so that using the fact $\gamma\ge0$ and \eqref{H}, we have
\begin{equation*}
\begin{split}
     &\left|\sum_{|\alpha|\le2}\sum_{j=1}^{3}\int_{\mathbb T^{3}_{x}}\int_{\Omega_{3}}a_{j1}(v-w)\mu(w)H^{m-1}_{\delta_{1}}H^{n}_{\delta_{2}}f(x, v)A_{-, j}H^{m}_{\delta_{1}}H^{n}_{\delta_{2}}f(x, v)dwdv\right|\\
     &\lesssim|||H^{m-1}_{\delta_{1}}H^{n}_{\delta_{2}}f|||_{H^{2}_{x}, \sigma}|||H^{m}_{\delta_{1}}H^{n}_{\delta_{2}}f|||_{H^{2}_{x}, \sigma}.
\end{split}
\end{equation*}
Combining the above results, it follows that
$$|I_{2, 1}|\lesssim mt^{\delta_{1}}|||H^{m-1}_{\delta_{1}}H^{n}_{\delta_{2}}f|||_{H^{2}_{x}, \sigma}|||H^{m}_{\delta_{1}}H^{n}_{\delta_{2}}f|||_{H^{2}_{x}, \sigma}.$$
Similarly, we can deduce
$$|I_{2, 2}|\lesssim nt^{\delta_{2}}|||H^{m}_{\delta_{1}}H^{n-1}_{\delta_{2}}f|||_{H^{2}_{x}, \sigma}|||H^{m}_{\delta_{1}}H^{n}_{\delta_{2}}f|||_{H^{2}_{x}, \sigma}.$$
So that
$$|I_{2}|\lesssim \left(mt^{\delta_{1}}|||H^{m-1}_{\delta_{1}}H^{n}_{\delta_{2}}f|||_{H^{2}_{x}, \sigma}+nt^{\delta_{2}}|||H^{m}_{\delta_{1}}H^{n-1}_{\delta_{2}}f|||_{H^{2}_{x}, \sigma}\right)|||H^{m}_{\delta_{1}}H^{n}_{\delta_{2}}f|||_{H^{2}_{x}, \sigma}.$$
For the term $I_{3}$, we can rewrite it as
\begin{equation*}
\begin{split}
     I_{3}&=\sum_{|\alpha|\le2}\sum_{j, k=1}^{3}\sum_{l=2}^{m}C_{m}^{l}t^{\delta_{1}l}(\partial^{l}_{v_{1}}\bar a_{jk}A_{-, k}H^{m-l}_{\delta_{1}}H^{n}_{\delta_{2}}\partial^{\alpha}_{x}f,\ A_{-, j}H^{m}_{\delta_{1}}H^{n}_{\delta_{2}}\partial^{\alpha}_{x}f)_{L^{2}_{x, v}}\\
     &\quad+\sum_{|\alpha|\le2}\sum_{j, k=1}^{3}\sum_{l=1}^{m}C_{m}^{l}t^{\delta_{1}l}(\partial^{l}_{v_{1}}\bar a_{jk}[H^{m-l}_{\delta_{1}}H^{n}_{\delta_{2}}, A_{-, k}]\partial^{\alpha}_{x}f,\ A_{-, j}H^{m}_{\delta_{1}}H^{n}_{\delta_{2}}\partial^{\alpha}_{x}f)_{L^{2}_{x, v}}\\
     &\quad+mt^{\delta_{1}}\sum_{|\alpha|\le2}\sum_{j, k=1}^{3}(\partial_{v_{1}}\bar a_{jk}A_{-, k}H^{m-1}_{\delta_{1}}H^{n}_{\delta_{2}}\partial^{\alpha}_{x}f,\ A_{-, j}H^{m}_{\delta_{1}}H^{n}_{\delta_{2}}\partial^{\alpha}_{x}f)_{L^{2}_{x, v}}\\
     &=I_{3, 1}+I_{3, 2}+I_{3, 3}.
\end{split}
\end{equation*}
From the Lemma 2.1 of~\cite{L-1}, one has
\begin{equation}\label{derivation a}
   |\partial^{l}_{v_{1}}\bar a_{jk}|\lesssim l^{\frac{[\gamma]}{2}+1}\sqrt{l!}\langle v\rangle^{\gamma}, \quad \forall \ l\ge2,
\end{equation}
then using Cauchy-Schwarz inequality, it follows that
\begin{equation*}
\begin{split}
     |I_{3, 1}|&\lesssim\sum_{l=2}^{m}C^{l}_{m}l^{\frac{[\gamma]}{2}+1}\sqrt{l!}t^{\delta_{1} l}|||H^{m-l}_{\delta_{1}}H^{n}_{\delta_{2}}f|||_{H^{2}_{x}, \sigma}|||H^{m}_{\delta_{1}}H^{n}_{\delta_{2}}f|||_{H^{2}_{x}, \sigma}.
\end{split}
\end{equation*}
Applying \eqref{derivation a} and the result of $I_{2}$, we can obtain that
\begin{equation*}
\begin{split}
     |I_{3, 2}|&\lesssim\sum_{l=1}^{m-1}C_{m}^{l}l^{\frac{[\gamma]}{2}+1}\sqrt{l!}(m-l)t^{\delta_{1}(l+1)}|||H^{m-l-1}_{\delta_{1}}H^{n}_{\delta_{2}}f|||_{H^{2}_{x}, \sigma}|||H^{m}_{\delta_{1}}H^{n}_{\delta_{2}}f|||_{H^{2}_{x}, \sigma}\\
     &\quad+\sum_{l=1}^{m}C_{m}^{l}l^{\frac{[\gamma]}{2}+1}\sqrt{l!}nt^{\delta_{1}l}t^{\delta_{2}}|||H^{m-l}_{\delta_{1}}H^{n-1}_{\delta_{2}}f|||_{H^{2}_{x}, \sigma}|||H^{m}_{\delta_{1}}H^{n}_{\delta_{2}}f|||_{H^{2}_{x}, \sigma}.
\end{split}
\end{equation*}
Decomposing $\mathbb R^{3}\times\mathbb R^{3}=\Omega_{1}\cup\Omega_{2}\cup\Omega_{3}$ as above, then we can obtain that
\begin{equation*}
\begin{split}
     |I_{3, 3}|&\lesssim mt^{\delta_{1}}|||H^{m-1}_{\delta_{1}}H^{n}_{\delta_{2}}f|||_{H^{2}_{x}, \sigma}|||H^{m}_{\delta_{1}}H^{n}_{\delta_{2}}f|||_{H^{2}_{x}, \sigma}.
\end{split}
\end{equation*}
And therefore,
\begin{equation*}
\begin{split}
     |I_{3}|&\lesssim \sum_{l=2}^{m}C^{l}_{m}l^{\frac{[\gamma]}{2}+1}\sqrt{l!}t^{\delta_{1} l}|||H^{m-l}_{\delta_{1}}H^{n}_{\delta_{2}}f|||_{H^{2}_{x}, \sigma}|||H^{m}_{\delta_{1}}H^{n}_{\delta_{2}}f|||_{H^{2}_{x}, \sigma}\\
     &\quad+\sum_{l=1}^{m-1}C_{m}^{l}l^{\frac{[\gamma]}{2}+1}\sqrt{l!}(m-l)t^{\delta_{1}(l+1)}|||H^{m-l-1}_{\delta_{1}}H^{n}_{\delta_{2}}f|||_{H^{2}_{x}, \sigma}|||H^{m}_{\delta_{1}}H^{n}_{\delta_{2}}f|||_{H^{2}_{x}, \sigma}\\
     &\quad+\sum_{l=1}^{m}C_{m}^{l}l^{\frac{[\gamma]}{2}+1}\sqrt{l!}nt^{\delta_{1}l}t^{\delta_{2}}|||H^{m-l}_{\delta_{1}}H^{n-1}_{\delta_{2}}f|||_{H^{2}_{x}, \sigma}|||H^{m}_{\delta_{1}}H^{n}_{\delta_{2}}f|||_{H^{2}_{x}, \sigma}\\
     &\quad+mt^{\delta_{1}}|||H^{m-1}_{\delta_{1}}H^{n}_{\delta_{2}}f|||_{H^{2}_{x}, \sigma}|||H^{m}_{\delta_{1}}H^{n}_{\delta_{2}}f|||_{H^{2}_{x}, \sigma}\\
     &\lesssim \sum_{l=1}^{m}C^{l}_{m}l^{\frac{[\gamma]}{2}+2}\sqrt{l!}t^{\delta_{1} l}|||H^{m-l}_{\delta_{1}}H^{n}_{\delta_{2}}f|||_{H^{2}_{x}, \sigma}|||H^{m}_{\delta_{1}}H^{n}_{\delta_{2}}f|||_{H^{2}_{x}, \sigma}\\
     &\quad+\sum_{l=1}^{m}C_{m}^{l}l^{\frac{[\gamma]}{2}+1}\sqrt{l!}nt^{\delta_{1}l}t^{\delta_{2}}|||H^{m-l}_{\delta_{1}}H^{n-1}_{\delta_{2}}f|||_{H^{2}_{x}, \sigma}|||H^{m}_{\delta_{1}}H^{n}_{\delta_{2}}f|||_{H^{2}_{x}, \sigma}.
\end{split}
\end{equation*}
Since $H_{\delta_{1}}H_{\delta_{2}}=H_{\delta_{2}}H_{\delta_1}$, as the argument in $I_{3}$, we have
\begin{equation*}
\begin{split}
     |I_{4}|&\lesssim \sum_{p=1}^{n}C^{p}_{n}p^{\frac{[\gamma]}{2}+2}\sqrt{p!}t^{\delta_{2} p}|||H^{m}_{\delta_{1}}H^{n-p}_{\delta_{2}}f|||_{H^{2}_{x}, \sigma}|||H^{m}_{\delta_{1}}H^{n}_{\delta_{2}}f|||_{H^{2}_{x}, \sigma}\\
     &\quad+\sum_{p=1}^{n}C_{n}^{p}p^{\frac{[\gamma]}{2}+1}\sqrt{p!}mt^{\delta_{1}}t^{\delta_{2}p}|||H^{m-1}_{\delta_{1}}H^{n-p}_{\delta_{2}}f|||_{H^{2}_{x}, \sigma}|||H^{m}_{\delta_{1}}H^{n}_{\delta_{2}}f|||_{H^{2}_{x}, \sigma}.
\end{split}
\end{equation*}
It remains to consider $I_{5}$. Noting that
\begin{equation*}
\begin{split}
     &(\partial^{p+l}_{v_{1}}\bar a_{jk}H^{m-l}_{\delta_{1}}H^{n-p}_{\delta_{2}}A_{-, k}\partial^{\alpha}_{x}f,\ A_{-, j}H^{m}_{\delta_{1}}H^{n}_{\delta_{2}}\partial^{\alpha}_{x}f)_{L^{2}_{x, v}}\\
     &=(\partial^{p+l}_{v_{1}}\bar a_{jk}A_{-, k}H^{m-l}_{\delta_{1}}H^{n-p}_{\delta_{2}}\partial^{\alpha}_{x}f,\ A_{-, j}H^{m}_{\delta_{1}}H^{n}_{\delta_{2}}\partial^{\alpha}_{x}f)_{L^{2}_{x, v}}\\
     &\quad+(\partial^{p+l}_{v_{1}}\bar a_{jk}[H^{m-l}_{\delta_{1}}H^{n-p}_{\delta_{2}}, A_{-, k}]\partial^{\alpha}_{x}f,\ A_{-, j}H^{m}_{\delta_{1}}H^{n}_{\delta_{2}}\partial^{\alpha}_{x}f)_{L^{2}_{x, v}}.
\end{split}
\end{equation*}
By using \eqref{derivation a}, then from Cauchy-Schwarz inequality
\begin{equation*}
\begin{split}
     &\left|(\partial^{p+l}_{v_{1}}\bar a_{jk}A_{-, k}H^{m-l}_{\delta_{1}}H^{n-p}_{\delta_{2}}\partial^{\alpha}_{x}f,\ A_{-, j}H^{m}_{\delta_{1}}H^{n}_{\delta_{2}}\partial^{\alpha}_{x}f)_{L^{2}_{x, v}}\right|\\
     &\lesssim (p+l)^{\frac{[\gamma]}{2}+1}\sqrt{(p+l)!}\|\langle \cdot\rangle^{\gamma/2}A_{-, k}H^{m-l}_{\delta_{1}}H^{n-p}_{\delta_{2}}\partial^{\alpha}_{x}f\|_{L^{2}_{x, v}}\|\langle \cdot\rangle^{\gamma/2}A_{-, j}H^{m}_{\delta_{1}}H^{n-p}_{\delta_{2}}\partial^{\alpha}_{x}f\|_{L^{2}_{x, v}}.
\end{split}
\end{equation*}
So that combining the result of $I_{2}$, we have
\begin{equation*}
\begin{split}
     |I_{5}|&\lesssim\sum_{l=1}^{m}\sum_{p=1}^{n}C_{n}^{p}C_{m}^{l}(p+l)^{\frac{[\gamma]}{2}+1}\sqrt{(p+l)!}t^{\delta_{1} l}t^{\delta_{2} p}|||H^{m-l}_{\delta_{1}}H^{n-p}_{\delta_{2}}f|||_{H^{2}_{x}, \sigma}|||H^{m}_{\delta_{1}}H^{n}_{\delta_{2}}f|||_{H^{2}_{x}, \sigma}\\
     &\quad+\sum_{l=1}^{m}\sum_{p=1}^{n}C_{n}^{p}C_{m}^{l}(p+l)^{\frac{[\gamma]}{2}+1}\sqrt{(p+l)!}t^{\delta_{1} l}t^{\delta_{2} p}\big((m-l)t^{\delta_{1}}|||H^{m-l-1}_{\delta_{1}}H^{n-p}_{\delta_{2}}f|||_{H^{2}_{x}, \sigma}\\
     &\quad+(n-p)t^{\delta_{2}}|||H^{m-l}_{\delta_{1}}H^{n-p-1}_{\delta_{2}}f|||_{H^{2}_{x}, \sigma}\big)|||H^{m}_{\delta_{1}}H^{n}_{\delta_{2}}f|||_{H^{2}_{x}, \sigma}\\
     &\lesssim\sum_{l=1}^{m}\sum_{p=1}^{n}C_{n}^{p}C_{m}^{l}(p+l)^{\frac{[\gamma]}{2}+2}\sqrt{(p+l)!}t^{\delta_{1} l}t^{\delta_{2} p}|||H^{m-l}_{\delta_{1}}H^{n-p}_{\delta_{2}}f|||_{H^{2}_{x}, \sigma}|||H^{m}_{\delta_{1}}H^{n}_{\delta_{2}}f|||_{H^{2}_{x}, \sigma}.
\end{split}
\end{equation*}
Combining these results, by using Cauchy-Schwarz inequality, there exists a constant $C_{4}>0$, independent of $m$ and $n$, such that
\begin{equation*}
\begin{split}
     &\left(H^{m}_{\delta_{1}}H^{n}_{\delta_{2}}\mathcal L_{1}f,\ H^{m}_{\delta_{1}}H^{n}_{\delta_{2}}f\right)_{H^{2}_{x}L^{2}_{v}}\ge\frac34|||H^{m}_{\delta_{1}}H^{n}_{\delta_{2}}f|||_{H^{2}_{x}, \sigma}^{2}-C_{4}\|\langle v\rangle^{\frac\gamma2}H^{m}_{\delta_{1}}H^{n}_{\delta_{2}}f\|^{2}_{H^{2}_{x}L^{2}_{v}}\\
     &-C_{4}\sum_{l=1}^{m}C^{l}_{m}l^{\frac{[\gamma]}{2}+2}\sqrt{l!}t^{\delta_{1} l}|||H^{m-l}_{\delta_{1}}H^{n}_{\delta_{2}}f|||_{H^{2}_{x}, \sigma}|||H^{m}_{\delta_{1}}H^{n}_{\delta_{2}}f|||_{H^{2}_{x}, \sigma}\\
     &-C_{4}\sum_{l=0}^{m}C_{m}^{l}l^{\frac{[\gamma]}{2}+1}\sqrt{l!}nt^{\delta_{1}l}t^{\delta_{2}}|||H^{m-l}_{\delta_{1}}H^{n-1}_{\delta_{2}}f|||_{H^{2}_{x}, \sigma}|||H^{m}_{\delta_{1}}H^{n}_{\delta_{2}}f|||_{H^{2}_{x}, \sigma}\\
     &-C_{4}\sum_{p=1}^{n}C^{p}_{n}p^{\frac{[\gamma]}{2}+2}\sqrt{p!}t^{\delta_{2} p}|||H^{m}_{\delta_{1}}H^{n-p}_{\delta_{2}}f|||_{H^{2}_{x}, \sigma}|||H^{m}_{\delta_{1}}H^{n}_{\delta_{2}}f|||_{H^{2}_{x}, \sigma}\\
     &-C_{4}\sum_{p=0}^{n}C_{n}^{p}p^{\frac{[\gamma]}{2}+1}\sqrt{p!}mt^{\delta_{1}}t^{\delta_{2}p}|||H^{m-1}_{\delta_{1}}H^{n-p}_{\delta_{2}}f|||_{H^{2}_{x}, \sigma}|||H^{m}_{\delta_{1}}H^{n}_{\delta_{2}}f|||_{H^{2}_{x}, \sigma}\\
     &-C_{4}\sum_{l=1}^{m}\sum_{p=1}^{n}C_{n}^{p}C_{m}^{l}(p+l)^{\frac{[\gamma]}{2}+2}\sqrt{(p+l)!}t^{\delta_{1} l}t^{\delta_{2} p}|||H^{m-l}_{\delta_{1}}H^{n-p}_{\delta_{2}}f|||_{H^{2}_{x}, \sigma}|||H^{m}_{\delta_{1}}H^{n}_{\delta_{2}}f|||_{H^{2}_{x}, \sigma}.
\end{split}
\end{equation*}
\end{proof}
\begin{remark}
     For $f\in H^{\infty}(\mathbb T^{3}_{x}; \mathcal S(\mathbb R^{3}_{v}))$ and $\gamma\ge 0$, then
\begin{equation}\label{L-1-0}
\begin{split}
     (\mathcal L_{1}f,\ f)_{H^{2}_{x}L^{2}_{v}}\ge|||f|||_{H^{2}_{x}, \sigma}^{2}-c_{0}\|\langle v\rangle^{\frac\gamma2}f\|_{H^{2}_{x}L^{2}_{v}}|||f|||_{H^{2}_{x}, \sigma}^{2}.
\end{split}
\end{equation}
and for all $\delta>\frac32$, $m\in\mathbb N$ and $0<\varepsilon<1$
\begin{equation}\label{L-1-n-0}
\begin{split}
     (H^{m}_{\delta}\mathcal L_{1}f,\ H^{m}_{\delta}f)_{H^{2}_{x}L^{2}_{v}}&\ge(1-\varepsilon)|||H^{m}_{\delta}f|||_{H^{2}_{x}, \sigma}^{2}-C_{\varepsilon}\|\langle v\rangle^{\frac\gamma2}H^{m}_{\delta}f\|^{2}_{H^{2}_{x}L^{2}_{v}}\\
     &\ -C_{\varepsilon}\left(\sum_{l=1}^{m}C^{l}_{m}l^{\frac{[\gamma]}{2}+2}\sqrt{l!}t^{\delta l}|||H^{m-l}_{\delta}f|||_{H^{2}_{x}, \sigma}\right)^{2}.
\end{split}
\end{equation}
\end{remark}

\section{Energy estimates for multi-directional derivations}\label{s4}

In this section, we establish the energy estimates for multi-directional derivations.
\begin{lemma}\label{lemma4.1}
    For $\gamma\ge0$. Let $f$ be the smooth solution of Cauchy problem \eqref{1-2} with $\|f\|_{L^{\infty}([0, \infty[; H^2_xL^{2}_v)}$ small enough. Then for all $T>0$ and $\delta>\frac32$, there exist $B>0, \tilde B>0$ such that
\begin{equation}\label{k=0}
      \|f(t)\|^2_{H^2_xL^{2}_v}+\int_0^t|||f(\tau)|||^2_{H^2_x, \sigma} d\tau\le B^{2}\epsilon^{2}, \quad \forall \ 0<t\le T,
\end{equation}
and
    \begin{equation}\label{k=1}
        \begin{split}
            &\|H_{\delta}f(t)\|^2_{H^{2}_{x}L^{2}_{v}}+\int_{0}^{t}|||H_{\delta}f(\tau)|||_{H^{2}_{x}, \sigma}^{2}d\tau\le \tilde B^{2}\epsilon^{2}, \quad \forall \ 0<t\le T.
        \end{split}
    \end{equation}
\end{lemma}
\begin{proof}
     Since $f$ is the smooth solution of Cauchy problem \eqref{1-2}, we have
     $$\frac{1}{2}\frac{d}{dt}\|f(t)\|^{2}_{H^{2}_{x}L^{2}_{v}}+(\mathcal L_{1}f, f)_{H^{2}_{x}L^{2}_{v}}=(\Gamma(f, f), f)_{H^{2}_{x}L^{2}_{v}}-(\mathcal L_{2}f, f)_{H^{2}_{x}L^{2}_{v}}.$$
Since $\mathcal L_{2}f=-\Gamma(f, \sqrt\mu)$, then from Lemma \ref{Gamma},
\begin{equation*}
\begin{split}
     &\left|(\mathcal L_{2}f, f)_{L^{2}_{v}}\right|\lesssim\|f\|_{L^{2}_{v}}|||f|||_{\sigma},
\end{split}
\end{equation*}
noting that $\partial^{\alpha}_{x}\mathcal L_{2}f=\mathcal L_{2}\partial^{\alpha}_{x}f$, then
\begin{equation*}
\begin{split}
     &\left|(\mathcal L_{2}f, f)_{H^{2}_{x}L^{2}_{v}}\right|\lesssim\sum_{|\alpha|\le2}\int_{\mathbb T^{3}_{x}}\|\partial^{\alpha}_{x}f\|_{L^{2}_{v}}|||\partial^{\alpha}_{x}f|||_{\sigma}dx\lesssim\|f\|_{H^{2}_{x}L^{2}_{v}}|||f|||_{H^{2}_{x}, \sigma}.
\end{split}
\end{equation*}
By using H${\rm {\ddot o}}$lder¡¯s inequality and corollary 1 of~\cite{G-1}, one has for all $g\in H^\infty(\mathbb T^3_x; \mathcal S(\mathbb R^3_v))$
\begin{equation*}
\begin{split}
     \|\langle v\rangle^{\frac\gamma2}g\|^{2}_{H^{2}_{x}L^{2}_{v}}&\le\sum_{|\alpha|\le2}\int_{\mathbb T^{3}_{x}}\|\langle\cdot\rangle^{\frac\gamma2+1}\partial^{\alpha}_{x}g(x, \cdot)\|^{\frac{2\gamma}{\gamma+2}}_{L^{2}_{v}}\|\partial^{\alpha}_{x}g(x, \cdot)\|^{\frac{4}{\gamma+2}}_{L^{2}_{v}}dx\\
     &\le\sum_{|\alpha|\le2}\int_{\mathbb T^{3}_{x}}\left(\frac{1}{\sqrt{C_{1}}}|||\partial^{\alpha}_{x}g(x, \cdot)|||_{\sigma}\right)^{\frac{2\gamma}{\gamma+2}}\|\partial^{\alpha}_{x}g(x, \cdot)\|^{\frac{4}{\gamma+2}}_{L^{2}_{v}}dx,
\end{split}
\end{equation*}
then from the Young inequality and the fact $\gamma\ge0$, we have that for all $0<\eta<1$
\begin{equation*}
\begin{split}
     &\left(\frac{1}{\sqrt{C_{1}}}|||\partial^{\alpha}_{x}g(x, \cdot)|||_{\sigma}\right)^{\frac{2\gamma}{\gamma+2}}\|\partial^{\alpha}_{x}g(x, \cdot)\|^{\frac{4}{\gamma+2}}_{L^{2}_{v}}\le\eta|||\partial^{\alpha}_{x}g(x, \cdot)|||^{2}_{\sigma}+C_{\eta}\|\partial^{\alpha}_{x}g(x, \cdot)\|^{2}_{L^{2}_{v}},
\end{split}
\end{equation*}
then we can get for all $g\in H^\infty(\mathbb T^3_x; \mathcal S(\mathbb R^3_v))$
\begin{equation}\label{interpolation}
     \|\langle v\rangle^{\frac\gamma2}g\|^{2}_{H^{2}_{x}L^{2}_{v}}\le\eta|||g|||^{2}_{H^{2}_{x}, \sigma}+C_{\eta}\|g\|^{2}_{H^{2}_{x}L^{2}_{v}},  \ \forall \ 0<\eta<1.
\end{equation}
So that applying \eqref{interpolation} with $g=f$ and $\eta=\frac18$, from \eqref{L-1-0} and Cauchy-Schwarz inequality, one has
$$(\mathcal L_{1}f,\ f)_{H^{2}_{x}L^{2}_{v}}\ge\frac34|||f|||_{H^{2}_{x}, \sigma}^{2}-\tilde c_{0}\|f\|^{2}_{H^{2}_{x}L^{2}_{v}}.$$
Therefore, combining the above inequalities and \eqref{Gamma-1}, one can get
\begin{equation*}
\begin{split}
     &\frac{1}{2}\frac{d}{dt}\|f(t)\|^{2}_{H^{2}_{x}L^{2}_{v}}+\frac34|||f|||^{2}_{H^{2}_{x}, \sigma}\\
     &\le\tilde c_{0}\|f\|^{2}_{H^{2}_{x}L^{2}_{v}}+C_{2}\|f\|_{H^{2}_{x}L^{2}_{v}}|||f|||^{2}_{H^{2}_{x}, \sigma}+\tilde C_{2}\|f\|_{H^{2}_{x}L^{2}_{v}}|||f|||_{H^{2}_{x}, \sigma}\\
     &\le C_{2}\|f\|_{H^{2}_{x}L^{2}_{v}}|||f|||^{2}_{H^{2}_{x}, \sigma}+(\tilde c_{0}+2(\tilde C_{2})^{2})\|f\|^{2}_{H^{2}_{x}L^{2}_{v}}+\frac18|||f|||^{2}_{H^{2}_{x}, \sigma}.
\end{split}
\end{equation*}
Since $\|f\|_{L^{\infty}([0, \infty[; H^2_xL^{2}_v)}$ small enough, for all $0<\epsilon<1$,
$$\|f\|_{L^{\infty}([0, \infty[; H^2_xL^{2}_v)}<\epsilon,$$
taking $C_{2}\epsilon\le\frac18$, we have
\begin{equation*}
\begin{split}
     &\frac{d}{dt}\|f(t)\|^{2}_{H^{2}_{x}L^{2}_{v}}+|||f|||^{2}_{H^{2}_{x}, \sigma}\le2(\tilde c_{0}+2(\tilde C_{2})^{2})\|f\|^{2}_{H^{2}_{x}L^{2}_{v}}.
\end{split}
\end{equation*}
For all $0<t\le T$, integrating from~0~ to $t$, it follows that
\begin{equation}\label{integration}
\begin{split}
     &\|f(t)\|^{2}_{H^{2}_{x}L^{2}_{v}}+\int_{0}^{t}|||f(\tau)|||^{2}_{H^{2}_{x}, \sigma}d\tau\\
     &\le\|f_{0}\|^{2}_{H^{2}_{x}L^{2}_{v}}+2(\tilde c_{0}+2(\tilde C_{2})^{2})\int_{0}^{t}\|f(\tau)\|^{2}_{H^{2}_{x}L^{2}_{x}}d\tau,
\end{split}
\end{equation}
 by Gronwall inequality, we get for all $0<t\le T$
$$\|f(t)\|^{2}_{H^{2}_{x}L^{2}_{v}}\le\left(1+2T(\tilde c_{0}+2(\tilde C_{2})^{2})e^{2T(\tilde c_{0}+2T(\tilde C_{2})^{2})}\right)\|f_{0}\|^{2}_{H^{2}_{x}L^{2}_{v}},$$
plugging it back into \eqref{integration}, and taking
$$B=\left(1+2T(\tilde c_{0}+2(\tilde C_{2})^{2})\right)e^{2T(\tilde c_{0}+2(\tilde C_{2})^{2})},$$
it follows that for all $0<t\le T$
\begin{equation*}
\begin{split}
     &\|f(t)\|^{2}_{H^{2}_{x}L^{2}_{v}}+\int_{0}^{t}|||f(\tau)|||^{2}_{H^{2}_{x}, \sigma}d\tau\le B^{2}\epsilon^{2}.
\end{split}
\end{equation*}
It remains to consider \eqref{k=1}. From \eqref{1-2} and \eqref{keyob}, we have
\begin{equation*}
        \begin{split}
            &\frac12\frac{d}{dt}\|H_{\delta}f\|^2_{H^2_xL^{2}_v}+(H_{\delta}\mathcal L_1f,  \ H_{\delta}f)_{H^2_xL^{2}_v}\\
            &=-\delta t^{\delta-1}( A_{+, 1}f, H_{\delta} f)_{H^2_xL^{2}_v}-(H_{\delta} \mathcal L_2f,  H_{\delta}f)_{H^2_xL^{2}_v}+(H_{\delta}\Gamma(f, f),  H_{\delta}f)_{H^2_xL^{2}_v}.
        \end{split}
    \end{equation*}
By using \eqref{L-1-n-0} with $\varepsilon=\frac18$, one has
\begin{equation*}
\begin{split}
     (H_{\delta}\mathcal L_{1}f,\ H_{\delta}f)_{H^{2}_{x}L^{2}_{v}}&\ge\frac78|||H_{\delta}f|||_{H^{2}_{x}, \sigma}^{2}-C_{5}\left\|\langle v\rangle^{\frac\gamma2}H_{\delta}f\right\|^{2}_{H^{2}_{x}L^{2}_{v}}-C_{5}t^{\delta}|||f|||^{2}_{H^{2}_{x}, \sigma},
\end{split}
\end{equation*}
applying \eqref{interpolation} with $g=H_\delta f$ and $C_{5}\eta=\frac18$, then it follows that
\begin{equation*}
\begin{split}
     (H_{\delta}\mathcal L_{1}f,\ H_{\delta}f)_{H^{2}_{x}L^{2}_{v}}&\ge\frac34|||H_{\delta}f|||_{H^{2}_{x}, \sigma}^{2}-\tilde C_{5}\left\|H_{\delta}f\right\|^{2}_{H^{2}_{x}L^{2}_{v}}-C_{5}t^{\delta}|||f|||^{2}_{H^{2}_{x}, \sigma}.
\end{split}
\end{equation*}
Since $\gamma\ge0$, by using \eqref{H}, we have
\begin{equation}\label{g1}
\left\|A_{+, 1}g(x, \cdot)\right\|^{2}_{L^2_{v}}\le \frac{1}{C_{1}}\left(|||g(x, \cdot)|||^{2}_{\sigma}+\left\|g(x, \cdot)\right\|^{2}_{L^2_{v}}\right), \quad \forall \ g\in H^\infty(\mathbb T^3_x; \mathcal S(\mathbb R^3_v)),
\end{equation}
applying \eqref{g1} with $g=\partial^\alpha_xf$, then we can obtain that
\begin{equation*}
\begin{split}
     &\left|\delta t^{\delta-1}( A_{+, 1}f, H_{\delta} f)_{H^2_xL^{2}_v}\right|\\
     &\le \delta t^{\delta-1}\sum_{|\alpha|\le2}\int_{\mathbb T^{3}_{x}}\|\partial^\alpha_x A_{+, 1}f(x, \cdot)\|_{L^{2}_v}\|H_{\delta}\partial^\alpha_x f(x, \cdot)\|_{L^{2}_v}d\tau\\
     &\le \frac{t^{2(\delta-1)}\delta^{2}}{4C_{1}}\left(|||f|||_{H^{2}_{x}, \sigma}^{2}+\|f\|^{2}_{H^{2}_{x}L^{2}_{v}}\right)+\|H_\delta f\|^2_{H^{2}_{x}L^2_{v}}.
\end{split}
\end{equation*}
Applying Proposition \ref{H-m-Gamma} with $\delta_{1}=\delta$ and $n=0$, it follows that
\begin{equation*}
\begin{split}
     |(H_{\delta}\Gamma(f, f),  H_{\delta}f)_{H^2_xL^{2}_v}|
     &\le C_{2}\|f\|_{H^{2}_{x}L^{2}_{v}}|||H_{\delta}f|||^{2}_{H^{2}_{x}, \sigma}\\
     &\quad+C_{2}\|H_{\delta}f\|_{H^{2}_{x}L^{2}_{v}}|||f|||_{H^{2}_{x}, \sigma}|||H_{\delta}f|||_{H^{2}_{x}, \sigma}.
\end{split}
\end{equation*}
By using Corollary \ref{corollary1} with $\delta_{1}=\delta$ and $n=0$, it follows that
\begin{equation*}
\begin{split}
     &|(H_{\delta}\mathcal L_{2}f,\ H_{\delta}f)_{H^{2}_{x}L^{2}_{v}}|\le C_{3}\sqrt{C_{0}}t^{\delta}\|f\|_{H^{2}_{x}L^{2}_{v}}|||H_{\delta}f|||_{H^{2}_{x}, \sigma}+C_{3}\|H_{\delta}f\|_{H^{2}_{x}L^{2}_{v}}|||H_{\delta}f|||_{H^{2}_{x}, \sigma}.
\end{split}
\end{equation*}
Combining these inequalities it follows that
     \begin{equation*}
        \begin{split}
            &\frac12\frac{d}{dt}\|H_{\delta}f\|^2_{H^2_xL^{2}_v}+\frac34|||H_{\delta}f|||_{H^{2}_{x}, \sigma}^{2}\\
            &\le (\tilde C_{5}+1)\|H_{\delta}f\|^{2}_{H^{2}_{x}L^{2}_{v}}+C_{5}t^{\delta}|||f|||^{2}_{H^{2}_{x}, \sigma}+\frac{t^{2(\delta-1)}\delta^{2}}{4C_{1}}\left(|||f|||_{H^{2}_{x}, \sigma}^{2}+\|f\|^{2}_{H^{2}_{x}L^{2}_{v}}\right)\\
            &\quad+C_{3}\sqrt{C_{0}}t^{\delta}\|f\|_{H^{2}_{x}L^{2}_{v}}|||H_{\delta}f|||_{H^{2}_{x}, \sigma}+C_{3}\|H_{\delta}f\|_{H^{2}_{x}L^{2}_{v}}|||H_{\delta}f|||_{H^{2}_{x}, \sigma}\\
            &\quad+C_{2}\|f\|_{H^{2}_{x}L^{2}_{v}}|||H_{\delta}f|||^{2}_{H^{2}_{x}, \sigma}+C_{2}\|H_{\delta}f\|_{H^{2}_{x}L^{2}_{v}}|||f|||_{H^{2}_{x}, \sigma}|||H_{\delta}f|||_{H^{2}_{x}, \sigma},
        \end{split}
    \end{equation*}
    for all $0<t\le T$, integrating from~0~ to $t$, then from \eqref{k=0} and Cauchy-Schwarz inequality, one can get
    \begin{equation*}
        \begin{split}
            &\frac12\|H_{\delta}f(t)\|^2_{H^2_xL^{2}_v}+\frac12\int_{0}^{t}|||H_{\delta}f(\tau)|||_{H^{2}_{x}, \sigma}^{2}d\tau\\
            &\le (\tilde C_{5}+1+(C_{3})^{2})\int_{0}^{t}\|H_{\delta}f(\tau)\|^{2}_{H^{2}_{x}L^{2}_{v}}d\tau+C_{2}B\epsilon\int_{0}^{t}|||H_{\delta}f(\tau)|||^{2}_{H^{2}_{x}, \sigma}d\tau\\
            &\quad +C_{3}\sqrt{C_{0}T}(T+1)^{\delta}B\epsilon\left(\int_{0}^{t}|||H_{\delta}f(\tau)|||^{2}_{H^{2}_{x}, \sigma}d\tau\right)^{\frac12}\\
            &\quad +C_{2}\int_{0}^{t}\|H_{\delta}f(\tau)\|_{H^{2}_{x}L^{2}_{v}}|||f(\tau)|||_{H^{2}_{x}, \sigma}|||H_{\delta}f(\tau)|||_{H^{2}_{x}, \sigma}d\tau\\
            &\quad+\left(C_{5}(T+1)^{\delta}+\frac{\delta^{2}(T+1)^{2\delta-1}}{4C_{1}}\right)B^{2}\epsilon^{2}.
        \end{split}
    \end{equation*}
    By using Cauchy-Schwarz inequality, one has
\begin{equation*}
\begin{split}
     &C_{3}\sqrt{C_{0}T}(T+1)^{\delta}B\epsilon\left(\int_{0}^{t}|||H_{\delta}f(\tau)|||^{2}_{H^{2}_{x}, \sigma}d\tau\right)^{\frac12}\\
     &\le\frac{1}{16}\int_{0}^{t}|||H_{\delta}f(\tau)|||^{2}_{H^{2}_{x}, \sigma}d\tau+4(T+1)^{2\delta}C_{0}T(C_{3})^{2}B^{2}\epsilon^{2},
\end{split}
\end{equation*}
and
\begin{equation*}
\begin{split}
     &C_{2}\int_{0}^{t}\|H_{\delta}f(\tau)\|_{H^{2}_{x}L^{2}_{v}}|||f(\tau)|||_{H^{2}_{x}, \sigma}|||H_{\delta}f(\tau)|||_{H^{2}_{x}, \sigma}d\tau\\
     &\le\frac{1}{16}\int_{0}^{t}|||H_{\delta}f(\tau)|||^{2}_{H^{2}_{x}, \sigma}d\tau+4(C_{2})^{2}B^{2}\epsilon^{2}\|H_{\delta}f\|^{2}_{L^{\infty}(]0, t]; H^{2}_{x}L^{2}_{v})}.
\end{split}
\end{equation*}
Combining these inequalities and taking $C_{2}B\epsilon\le\frac18$, it follows that for all $0<t\le T$
    \begin{equation*}
        \begin{split}
            &\|H_{\delta}f(t)\|^2_{H^2_xL^{2}_v}+\frac12\int_{0}^{t}|||H_{\delta}f(\tau)|||_{H^{2}_{x}, \sigma}^{2}d\tau\\
            &\le 2(\tilde C_{5}+1+(C_{3})^{2})\int_{0}^{t}\|H_{\delta}f(\tau)\|^{2}_{H^{2}_{x}L^{2}_{v}}d\tau+2\tilde C_{3}B^{2}\epsilon^{2}+\frac{1}{16}\|H_{\delta}f\|^{2}_{L^{\infty}(]0, t]; H^{2}_{x}L^{2}_{v})},
        \end{split}
    \end{equation*}
    with
    $$\tilde C_{3}=4(T+1)^{2\delta}C_{0}T(C_{3})^{2}+C_{5}(T+1)^{\delta}+\frac{\delta^{2}(T+1)^{2\delta-1}}{4C_{1}}.$$
    Then we can get that for all $T>0$
    \begin{equation*}
        \begin{split}
            &\frac12\|H_{\delta}f\|^2_{L^{\infty}(]0, T]; H^{2}_{x}L^{2}_{v})}\\
            &\le 2(\tilde C_{5}+1+(C_{3})^{2})\int_{0}^{T}\|H_{\delta}f(\tau)\|^{2}_{H^{2}_{x}L^{2}_{v}}d\tau+2\tilde C_{3}B^{2}\epsilon^{2},
        \end{split}
    \end{equation*}
    so that for all $0<t\le T$
    \begin{equation}\label{4-5}
        \begin{split}
            &\|H_{\delta}f(t)\|^2_{H^2_xL^{2}_v}+\int_{0}^{t}|||H_{\delta}f(\tau)|||_{H^{2}_{x}, \sigma}^{2}d\tau\\
            &\le 4(\tilde C_{5}+1+(C_{3})^{2})\int_{0}^{t}\|H_{\delta}f(\tau)\|^{2}_{H^{2}_{x}L^{2}_{v}}d\tau+4\tilde C_{3}B^{2}\epsilon^{2}.
        \end{split}
    \end{equation}
    By using Gronwall inequality, we have for all $0 < t \le T$
    \begin{equation*}
    \begin{split}
         \|H_{\delta}f(t)\|^2_{H^2_xL^{2}_v}&\le4\tilde C_{3}B^{2}\epsilon^{2}\left(1+4T(\tilde C_{5}+1+(C_{3})^{2})e^{4T(\tilde C_{5}+1+(C_{3})^{2})}\right),
    \end{split}
    \end{equation*}
    plugging it back into \eqref{4-5}, then we can get that for all $0<t\le T$
    \begin{equation*}
        \begin{split}
            &\|H_{\delta}f(t)\|^2_{H^2_xL^{2}_v}+\int_{0}^{t}|||H_{\delta}f(\tau)|||_{H^{2}_{x}, \sigma}^{2}d\tau\le \tilde B^{2}\epsilon^{2},
        \end{split}
    \end{equation*}
here
$$\tilde B\ge2B\left(1+4T(\tilde C_{5}+1+(C_{3})^{2})e^{4T(\tilde C_{5}+1+(C_{3})^{2})}\right)\sqrt{\tilde C_{3}}.$$
\end{proof}

\begin{prop}\label{prop 4.2}
    For $\gamma\ge0$. Let $f$ be the smooth solution of Cauchy problem \eqref{1-2} with $\|f\|_{L^{\infty}([0, \infty[; H^2_xL^{2}_v)}$ small enough. Then for all $T>0$ and $\delta_{1}, \delta_{2}>\frac32$, there exists a constant $A>0$ such that for any $0<t\le T$ and $m\in\mathbb N$, $n\in\mathbb N_{+}$
    \begin{equation}\label{4-1-mn}
    \begin{split}
       &\|H^{m}_{\delta_{1}}H^{n}_{\delta_{2}}f(t)\|^2_{H^2_xL^{2}_v}+\int_0^t||| H^{m}_{\delta_{1}}H^{n}_{\delta_{2}}f(\tau)|||^2_{H^2_x, \sigma} d\tau\\
       &\le \left(A^{m+n-\frac 12}(m-2)!(n-2)!\right)^2,
    \end{split}
    \end{equation}
    here the constant $A$ depends on $\gamma, \delta_{1}, \delta_{2}, C_{0}-C_{4}$ and $T$.
\end{prop}

We prove this proposition by induction on index $m+n=k$. For the case of $m+n=k=1$, taking $(\tilde B\epsilon)^{2}\le A$, then from \eqref{k=1}, one can deduce \eqref{4-1-mn} holds. By convention, we denote $k!=1$ if $k\le 0$. Assume $k\ge2$, and for all $1\le m+n\le k-1$, $m, n\in\mathbb N$,
\begin{equation}\label{hypothesis k-1}
\begin{split}
     &\|H^{m}_{\delta_{1}}H^{n}_{\delta_{2}}f(t)\|^2_{H^2_xL^{2}_v}+\int_0^t||| H^{m}_{\delta_{1}}H^{n}_{\delta_{2}}f(\tau)|||^2_{H^2_x, \sigma} d\tau\\
     &\le \left( A^{m+n-\frac 12 }(m-2)!(n-2)!\right)^2, \quad \forall \ 0< t\le T, \ \delta_{1}, \ \delta_{2}>\frac32.
\end{split}
\end{equation}

      Now, we would prove that \eqref{hypothesis k-1} holds for all $m, n\in\mathbb N$ with $m+n=k$.

      From \eqref{1-2}, we have
\begin{equation*}
\begin{split}
     &(\partial_t+v\ \cdot\ \partial_x)H^{m}_{\delta_{1}}H^{n}_{\delta_{2}}f+[H^{m}_{\delta_{1}}H^{n}_{\delta_{2}}, \partial_t+v\ \cdot\ \partial_x]f+H^{m}_{\delta_{1}}H^{n}_{\delta_{2}}\mathcal L_{1}f\\
     &=H^{m}_{\delta_{1}}H^{n}_{\delta_{2}}\Gamma(f, f)-H^{m}_{\delta_{1}}H^{n}_{\delta_{2}}\mathcal L_{2}f,
\end{split}
\end{equation*}
then taking the scalar product in $H^2_xL^{2}_v$ with respect to $H^{m}_{\delta_1}H^n_{\delta_2}f$, we can obtain
\begin{equation*}
\begin{split}
     &\frac12\frac{d}{dt}\left\|H^{m}_{\delta_{1}}H^{n}_{\delta_{2}}f\right\|^{2}_{H_{x}^{2}L^{2}_{v}}+\left(H^{m}_{\delta_{1}}H^{n}_{\delta_{2}}\mathcal L_{1}f, H^{m}_{\delta_{1}}H^{n}_{\delta_{2}}f\right)_{H^{2}_{x}L^{2}_{v}}\\
     &=-\left(\left[H^{m}_{\delta_{1}}H^{n}_{\delta_{2}}, \partial_t+v\ \cdot\ \partial_x\right]f, H^{m}_{\delta_{1}}H^{n}_{\delta_{2}}f\right)_{H^{2}_{x}L^{2}_{v}}-\left(H^{m}_{\delta_{1}}H^{n}_{\delta_{2}}\mathcal L_{2}f, H^{m}_{\delta_{1}}H^{n}_{\delta_{2}}f\right)_{H^{2}_{x}L^{2}_{v}}\\
     &\quad+\left(H^{m}_{\delta_{1}}H^{n}_{\delta_{2}}\Gamma(f, f), H^{m}_{\delta_{1}}H^{n}_{\delta_{2}}f\right)_{H^{2}_{x}L^{2}_{v}}.
\end{split}
\end{equation*}
If $m=0, n=k$ or $n=0, m=k$, then the commutator in the above formula has been given in \eqref{kehigher}. If $m, n\ge1$, then by applying \eqref{kehigher}, we can deduce that
\begin{equation*}
\begin{split}
     &[H^{m}_{\delta_{1}}H^{n}_{\delta_{2}}, \partial_t+v\ \cdot\ \partial_x]\\
     &=H^{n}_{\delta_{2}}[H^{m}_{\delta_{1}}, \partial_t+v\ \cdot\ \partial_x]+[H^{n}_{\delta_{2}}, \partial_t+v\ \cdot\ \partial_x]H^{m}_{\delta_{1}}\\
     &=\delta_{1}mt^{\delta_{1}-1}A_{+, 1}H^{m-1}_{\delta_{1}}H^{n}_{\delta_{2}}+\delta_{2}nt^{\delta_{2}-1}A_{+, 1}H^{m}_{\delta_{1}}H^{n-1}_{\delta_{2}},
\end{split}
\end{equation*}
For simplicity of the presentation, we consider the case of $m, n\ge1$ with $m+n=k$. In the case of $m=0, n=k$ and $n=0, m=k$, the proof is similar and quite easy.

Applying \eqref{g1} and Cauchy-Schwarz inequality, we can obtain that
\begin{equation*}
\begin{split}
     &\left|\left(\left[H^{m}_{\delta_{1}}H^{n}_{\delta_{2}}, \partial_t+v\ \cdot\ \partial_x\right]f, H^{m}_{\delta_{1}}H^{n}_{\delta_{2}}f\right)_{H^{2}_{x}L^{2}_{v}}\right|\\
     &\le\frac{(\delta_{1}mt^{\delta_{1}-1})^{2}}{8C_{1}}\left(|||H^{m-1}_{\delta_{1}}H^{n}_{\delta_{2}}f|||^{2}_{H^{2}_{x}, \sigma}+\|H^{m-1}_{\delta_{1}}H^{n}_{\delta_{2}}f\|^{2}_{H^{2}_{x}L^{2}_{v}}\right)+\|H^{m}_{\delta_{1}}H^{n}_{\delta_{2}}f\|^{2}_{H^{2}_{x}L^{2}_{v}}\\
     &\quad+\frac{(\delta_{2}nt^{\delta_{2}-1})^{2}}{8C_{1}}\left(|||H^{m}_{\delta_{1}}H^{n-1}_{\delta_{2}}f|||^{2}_{H^{2}_{x}, \sigma}+\|H^{m}_{\delta_{1}}H^{n-1}_{\delta_{2}}f\|^{2}_{H^{2}_{x}L^{2}_{v}}\right).
\end{split}
\end{equation*}
Since $\gamma\ge0$, then from corollary 1 of~\cite{G-1}, we have
\begin{equation*}
\begin{split}
     &\|H^{m-1}_{\delta_{1}}H^{n}_{\delta_{2}}f\|^{2}_{H^{2}_{x}L^{2}_{v}}
     \le\frac{1}{C_{1}}||| H^{m-1}_{\delta_{1}}H^{n}_{\delta_{2}}f|||^{2}_{H^{2}_{x}, \sigma},
\end{split}
\end{equation*}
and
\begin{equation*}
\begin{split}
     &\|H^{m}_{\delta_{1}}H^{n-1}_{\delta_{2}}f\|^{2}_{H^{2}_{x}L^{2}_{v}}
     \le\frac{1}{C_{1}}||| H^{m}_{\delta_{1}}H^{n-1}_{\delta_{2}}f|||^{2}_{H^{2}_{x}, \sigma},
\end{split}
\end{equation*}
so that
\begin{equation*}
\begin{split}
     &\left|\left([H^{m}_{\delta_{1}}H^{n}_{\delta_{2}}, \partial_t+v\ \cdot\ \partial_x]f, H^{m}_{\delta_{1}}H^{n}_{\delta_{2}}f\right)_{H^{2}_{x}L^{2}_{v}}\right|
     \le\|H^{m}_{\delta_{1}}H^{n}_{\delta_{2}}f\|^{2}_{H^{2}_{x}L^{2}_{v}}+R_{1}(t),
\end{split}
\end{equation*}
with
\begin{equation}\label{R-1}
\begin{split}
     R_{1}(t)&=\left(1+\frac{1}{C_{1}}\right)\frac{(\delta_{1}mt^{\delta_{1}-1})^{2}}{8C_{1}}|||H^{m-1}_{\delta_{1}}H^{n}_{\delta_{2}}f|||^{2}_{H^{2}_{x}, \sigma}\\
     &\quad+\left(1+\frac{1}{C_{1}}\right)\frac{(\delta_{2}nt^{\delta_{2}-1})^{2}}{8C_{1}}|||H^{m}_{\delta_{1}}H^{n-1}_{\delta_{2}}f|||^{2}_{H^{2}_{x}, \sigma}.
\end{split}
\end{equation}
Since $m, n\ge1$, from Proposition \ref{H-m-Gamma}, we can obtain that
\begin{equation*}
\begin{split}
     &\left|\left(H^{m}_{\delta_{1}}H^{n}_{\delta_{2}}\Gamma(f, f),\ H^{m}_{\delta_{1}}H^{n}_{\delta_{2}}f\right)_{H^{2}_{x}L^{2}_{v}}\right|\\
     &\le C_{2}\|f\|_{H^{2}_{x}L^{2}_{v}}|||H^{m}_{\delta_{1}}H^{n}_{\delta_{2}}f|||^{2}_{H^{2}_{x}, \sigma}\\
     &\qquad+C_{2}\|H^{m}_{\delta_{1}}H^{n}_{\delta_{2}}f\|_{H^{2}_{x}L^{2}_{v}}|||f|||_{H^{2}_{x}, \sigma}|||H^{m}_{\delta_{1}}H^{n}_{\delta_{2}}f|||_{H^{2}_{x}, \sigma}+R_{2}(t),
\end{split}
\end{equation*}
with
\begin{equation}\label{R-2}
\begin{split}
     R_{2}(t)&=C_{2}\sum_{l=1}^{m}\sum_{p=0}^{n-1}C_{m}^{l}C_{n}^{p}\|H^{l}_{\delta_{1}}
     H^{p}_{\delta_{2}}f\|_{H^{2}_{x}L^{2}_{v}}\\
     &\qquad\qquad\qquad\qquad\times |||H^{m-l}_{\delta_{1}}H^{n-p}_{\delta_{2}}f|||_{H^{2}_{x}, \sigma}|||H^{m}_{\delta_{1}}H^{n}_{\delta_{2}}f|||_{H^{2}_{x}, \sigma}\\
     &\quad+C_{2}\sum_{l=1}^{m-1}C_{m}^{l}\|H^{l}_{\delta_{1}}H^{n}_{\delta_{2}}f\|_{H^{2}_{x}L^{2}_{v}}|||H^{m-l}_{\delta_{1}}f|||_{H^{2}_{x}, \sigma}|||H^{m}_{\delta_{1}}H^{n}_{\delta_{2}}f|||_{H^{2}_{x}, \sigma}\\
     &\quad+C_{2}\sum_{p=1}^{n}C_{n}^{p}\|H^{p}_{\delta_{2}}f\|_{H^{2}_{x}L^{2}_{v}}|||H^{m}_{\delta_{1}}H^{n-p}_{\delta_{2}}f|||_{H^{2}_{x}, \sigma}|||H^{m}_{\delta_{1}}H^{n}_{\delta_{2}}f|||_{H^{2}_{x}, \sigma}.
\end{split}
\end{equation}
Since $m, n\ge1$, from Corollary \ref{corollary1}, by using Cauchy-Schwarz inequality, we have
\begin{equation*}
\begin{split}
     &|(H^{m}_{\delta_{1}}H^{n}_{\delta_{2}}\mathcal L_{2}f,\ H^{m}_{\delta_{1}}H^{n}_{\delta_{2}}f)_{H^{2}_{x}L^{2}_{v}}|\\
     &\le4(C_{3})^{2}\|H^{m}_{\delta_{1}}H^{n}_{\delta_{2}}f\|^{2}_{H^{2}_{x}L^{2}_{v}}+\frac{1}{16}|||H^{m}_{\delta_{1}}H^{n}_{\delta_{2}}f|||^{2}_{H^{2}_{x}, \sigma}+R_{3}(t),
\end{split}
\end{equation*}
with
\begin{equation}\label{R-3}
\begin{split}
     &R_{3}(t)=C_{3}\sum_{l=0}^{m-1}\sum_{p=0}^{n}C_{m}^{l}C^{p}_{n}\left(\sqrt{C_{0}}t^{\delta_{1}}\right)^{m-l}\left(\sqrt{C_{0}}t^{\delta_{2}}\right)^{n-p}\\
     &\qquad \qquad\quad\times\sqrt{(m-l+n-p)!}\|H^{l}_{\delta_{1}}H^{p}_{\delta_{2}}f\|_{H^{2}_{x}L^{2}_{v}}|||H^{m}_{\delta_{1}}H^{n}_{\delta_{2}}f|||_{H^{2}_{x}, \sigma}\\
     &\quad+C_{3}\sum_{p=0}^{n-1}C^{p}_{n}\left(\sqrt{C_{0}}t^{\delta_{2}}\right)^{n-p}\sqrt{(n-p)!}\|H^{m}_{\delta_{1}}H^{p}_{\delta_{2}}f\|_{H^{2}_{x}L^{2}_{v}}|||H^{m}_{\delta_{1}}H^{n}_{\delta_{2}}f|||_{H^{2}_{x}, \sigma}.
\end{split}
\end{equation}
Applying \eqref{interpolation} with $g=H^{m}_{\delta_{1}}H^{n}_{\delta_{2}}f$ and $C_{4}\eta=\frac{1}{16}$, one has
$$C_{4}\|\langle v\rangle^{\frac\gamma2}H^{m}_{\delta_{1}}H^{n}_{\delta_{2}}f\|^{2}_{H^{2}_{x}L^{2}_{v}}\le\frac{1}{16}|||H^{m}_{\delta_{1}}H^{n}_{\delta_{2}}f|||^{2}_{H^{2}_{x}, \sigma}+\tilde C_{4}\|H^{m}_{\delta_{1}}H^{n}_{\delta_{2}}f\|^{2}_{H^{2}_{x}L^{2}_{v}},$$
so that from Proposition \ref{H-m-L-1}, we can get that
\begin{equation*}
\begin{split}
&\left(H^{m}_{\delta_{1}}H^{n}_{\delta_{2}}\mathcal L_{1}f,\ H^{m}_{\delta_{1}}H^{n}_{\delta_{2}}f\right)_{H^{2}_{x}L^{2}_{v}}\\
     &\ge\frac{11}{16}|||H^{m}_{\delta_{1}}H^{n}_{\delta_{2}}f|||_{H^{2}_{x}, \sigma}^{2}-\tilde C_{4}\|H^{m}_{\delta_{1}}H^{n}_{\delta_{2}}f\|^{2}_{H^{2}_{x}L^{2}_{v}}-R_{4}(t),
\end{split}
\end{equation*}
with
\begin{equation}\label{R-4}
\begin{split}
&R_{4}(t)=C_{4}\sum_{l=0}^{m}C_{m}^{l}l^{\frac{[\gamma]}{2}+1}\sqrt{l!}nt^{\delta_{1}l}t^{\delta_{2}}|||H^{m-l}_{\delta_{1}}H^{n-1}_{\delta_{2}}f|||_{H^{2}_{x}, \sigma}|||H^{m}_{\delta_{1}}H^{n}_{\delta_{2}}f|||_{H^{2}_{x}, \sigma}\\
     &\quad+C_{4}\sum_{p=1}^{n}C^{p}_{n}p^{\frac{[\gamma]}{2}+2}\sqrt{p!}t^{\delta_{2} p}|||H^{m}_{\delta_{1}}H^{n-p}_{\delta_{2}}f|||_{H^{2}_{x}, \sigma}|||H^{m}_{\delta_{1}}H^{n}_{\delta_{2}}f|||_{H^{2}_{x}, \sigma}\\
     &\quad+C_{4}\sum_{p=0}^{n}C_{n}^{p}p^{\frac{[\gamma]}{2}+1}\sqrt{p!}mt^{\delta_{1}}t^{\delta_{2}p}|||H^{m-1}_{\delta_{1}}H^{n-p}_{\delta_{2}}f|||_{H^{2}_{x}, \sigma}|||H^{m}_{\delta_{1}}H^{n}_{\delta_{2}}f|||_{H^{2}_{x}, \sigma}\\
     &\quad+C_{4}\sum_{l=0}^{m}\sum_{p=1}^{n}C_{n}^{p}C_{m}^{l}(p+l)^{\frac{[\gamma]}{2}+2}\sqrt{(p+l)!}\\
     &\qquad\qquad\times t^{\delta_{1} l}t^{\delta_{2} p}|||H^{m-l}_{\delta_{1}}H^{n-p}_{\delta_{2}}f|||_{H^{2}_{x}, \sigma}|||H^{m}_{\delta_{1}}H^{n}_{\delta_{2}}f|||_{H^{2}_{x}, \sigma}.
\end{split}
\end{equation}
Combining the above results, it follows that
\begin{equation*}
\begin{split}
     &\frac{d}{dt}\|H^{m}_{\delta_{1}}H^{n}_{\delta_{2}}f\|^{2}_{H_{x}^{2}L^{2}_{v}}+\frac54|||H^{m}_{\delta_{1}}H^{n}_{\delta_{2}}f|||_{H^{2}_{x}, \sigma}^{2}\\
     &\le 2(1+\tilde C_{4}+4(C_{3})^{2})\|H^{m}_{\delta_{1}}H^{n}_{\delta_{2}}f\|^{2}_{H^{2}_{x}L^{2}_{v}}+2C_{2}\|f\|_{H^{2}_{x}L^{2}_{v}}|||H^{m}_{\delta_{1}}H^{n}_{\delta_{2}}f|||^{2}_{H^{2}_{x}, \sigma}\\
     &\quad+2C_{2}\|H^{m}_{\delta_{1}}H^{n}_{\delta_{2}}f\|_{H^{2}_{x}L^{2}_{v}}|||f|||_{H^{2}_{x}, \sigma}|||H^{m}_{\delta_{1}}H^{n}_{\delta_{2}}f|||_{H^{2}_{x}, \sigma}\\
     &\quad+2R_{1}(t)+2R_{2}(t)+2R_{3}(t)+2R_{4}(t).
\end{split}
\end{equation*}
For all $0<t\le T$, integrating from~0~ to $t$, since $\|f\|_{L^{\infty}([0, \infty[; H^2_xL^{2}_v)}$ small enough,
 it follows that for all $0<\epsilon<1$
\begin{equation}\label{integration-k}
\begin{split}
    &\|H^{m}_{\delta_{1}}H^{n}_{\delta_{2}}f(t)\|^{2}_{H_{x}^{2}L^{2}_{v}}+\frac54\int_{0}^{t}|||H^{m}_{\delta_{1}}H^{n}_{\delta_{2}}f(\tau)|||_{H^{2}_{x}, \sigma}^{2}d\tau\\
    &\le C_{6}\int_{0}^{t}\|H^{m}_{\delta_{1}}H^{n}_{\delta_{2}}f(\tau)\|^{2}_{H^{2}_{x}L^{2}_{v}}d\tau+2C_{2}\epsilon\int_{0}^{t}|||H^{m}_{\delta_{1}}H^{n}_{\delta_{2}}f|||^{2}_{H^{2}_{x}, \sigma}d\tau\\
    &\quad+2C_{2}\int_{0}^{t}\|H^{m}_{\delta_{1}}H^{n}_{\delta_{2}}f(\tau)\|_{H^{2}_{x}L^{2}_{v}}|||f(\tau)|||_{H^{2}_{x}, \sigma}|||H^{m}_{\delta_{1}}H^{n}_{\delta_{2}}f(\tau)|||_{H^{2}_{x}, \sigma}d\tau\\
    &\quad+\int_0^tR_1(\tau)d\tau+\int_0^tR_2(\tau)d\tau+\int_0^tR_3(\tau)d\tau+\int_0^tR_4(\tau)d\tau,
\end{split}
\end{equation}
here $C_{6}=2(1+\tilde C_{4}+4(C_{3})^{2})$.

We estimate the terms of the right-hand side of \eqref{integration-k} by the following lemmas.
\begin{lemma}\label{lemma5.2}
     Assume that $f$ satisfies \eqref{k=0}, then for all $0<t\le T$
\begin{equation*}
\begin{split}
     &2C_{2}\int_{0}^{t}\|H^{m}_{\delta_{1}}H^{n}_{\delta_{2}}f(\tau)\|_{H^{2}_{x}L^{2}_{v}}|||f(\tau)|||_{H^{2}_{x}, \sigma}|||H^{m}_{\delta_{1}}H^{n}_{\delta_{2}}f(\tau)|||_{H^{2}_{x}, \sigma}d\tau\\
     &\le16(C_{2}B\epsilon)^{2}\|H^{m}_{\delta_{1}}H^{n}_{\delta_{2}}f\|^{2}_{L^{\infty}(]0, t]; H^{2}_{x}L^{2}_{v})}+\frac{1}{16}\int_{0}^{t}|||H^{m}_{\delta_{1}}H^{n}_{\delta_{2}}f(\tau)|||^{2}_{H^{2}_{x}, \sigma}d\tau.
\end{split}
\end{equation*}
\end{lemma}
\begin{lemma}\label{lemma5.3}
     Assume that $f$ satisfies \eqref{hypothesis k-1} for $1\le m+n\le k-1$, then for all $0<t\le T$ and $m+ n=k$, we have
\begin{equation}\label{R-1+}
\int_{0}^{t}R_{1}(\tau)d\tau\le A_{1}(A^{m+n-\frac32}(m-2)!(n-2)!)^{2},
\end{equation}
with $A_{1}$ depends on $C_{1}, \delta_{1}, \delta_{2}, T$, and
\begin{equation}\label{R-2+}
\int_{0}^{t}R_{2}(\tau)d\tau\le A_{2}(A^{m+n-1}(m-2)!(n-2)!)^{2}+\frac{1}{16}\int_{0}^{t}|||H^{m}_{\delta_{1}}H^{n}_{\delta_{2}}f(\tau)|||^{2}_{H^{2}_{x}, \sigma}d\tau,
\end{equation}
with $A_{2}$ depends on $C_{2}, \delta_{1}, \delta_{2}, T$, and
\begin{equation}\label{R-3+}
\int_{0}^{t}R_{3}(\tau)d\tau\le A_{3}(A^{m+n-1}(m-2)!(n-2)!)^{2}+\frac{1}{32}\int_{0}^{t}|||H^{m}_{\delta_{1}}H^{n}_{\delta_{2}}f(\tau)|||^{2}_{H^{2}_{x}, \sigma}d\tau,
\end{equation}
with $A_{3}$ depends on $C_{0}, C_{3}, \delta_{1}, \delta_{2}, T$, and
\begin{equation}\label{R-4+}
\int_{0}^{t}R_{4}(\tau)d\tau\le A_{4}(A^{m+n-1}(m-2)!(n-2)!)^{2}+\frac{1}{32}\int_{0}^{t}|||H^{m}_{\delta_{1}}H^{n}_{\delta_{2}}f(\tau)|||^{2}_{H^{2}_{x}, \sigma}d\tau,
\end{equation}
with $A_{4}$ depends on $C_{4}, \gamma, \delta_{1}, \delta_{2}, T$, where $R_1(t), \cdots, R_4(t)$ are defined in \eqref{R-1}-\eqref{R-4}.
\end{lemma}
We will give the proofs of these lemmas in the next section.

\bigskip
\noindent
{\bf End of Proof of Proposition \ref{prop 4.2}:}

Choose $0<\epsilon<1$ such that
$$2C_{2}\epsilon\le\frac{1}{16}, \qquad 16(C_{2}B\epsilon)^{2}\le\frac12.$$
Since \eqref{hypothesis k-1} holds true for any $1\le m+n\le k-1$, and \eqref{k=0}, then combine the results of Lemma \ref{lemma5.2}-Lemma \ref{lemma5.3}, we get for all $0<t\le T$
\begin{equation*}
\begin{split}
    &\|H^{m}_{\delta_{1}}H^{n}_{\delta_{2}}f(t)\|^{2}_{H_{x}^{2}L^{2}_{v}}+\int_{0}^{t}|||H^{m}_{\delta_{1}}H^{n}_{\delta_{2}}f(\tau)|||_{H^{2}_{x}, \sigma}^{2}d\tau\\
    &\le C_{6}\int_{0}^{t}\|H^{m}_{\delta_{1}}H^{n}_{\delta_{2}}f(\tau)\|^{2}_{H^{2}_{x}L^{2}_{v}}d\tau+\frac12\|H^{m}_{\delta_{1}}H^{n}_{\delta_{2}}f\|^{2}_{L^{\infty}(]0, t]; H^{2}_{x}L^{2}_{v})}\\
    &\quad+A_{0}(A^{m+n-1}(m-2)!(n-2)!)^{2},
\end{split}
\end{equation*}
if we choose $A\ge1$, here $A_{0}=A_{1}+A_{2}+A_{3}+A_{4}$, from this, we have for all $T>0$
\begin{equation*}
\begin{split}
    &\frac12\|H^{m}_{\delta_{1}}H^{n}_{\delta_{2}}f\|^{2}_{L^{\infty}(]0, T]; H^{2}_{x}L^{2}_{v})}\le C_{6}\int_{0}^{T}\|H^{m}_{\delta_{1}}H^{n}_{\delta_{2}}f(\tau)\|^{2}_{H^{2}_{x}L^{2}_{v}}d\tau\\
    &\qquad+A_{0}(A^{m+n-1}(m-2)!(n-2)!)^{2},
\end{split}
\end{equation*}
so that, we have for all $0<t\le T$
\begin{equation}\label{integration k}
\begin{split}
    &\|H^{m}_{\delta_{1}}H^{n}_{\delta_{2}}f(t)\|^{2}_{H_{x}^{2}L^{2}_{v}}+\int_{0}^{t}|||H^{m}_{\delta_{1}}H^{n}_{\delta_{2}}f(\tau)|||_{H^{2}_{x}, \sigma}^{2}d\tau\\
    &\le 2C_{6}\int_{0}^{t}\|H^{m}_{\delta_{1}}H^{n}_{\delta_{2}}f(\tau)\|^{2}_{H^{2}_{x}L^{2}_{v}}d\tau+2A_{0}(A^{m+n-1}(m-2)!(n-2)!)^{2}.
\end{split}
\end{equation}
Using Gronwall inequality, one has for all $0<t\le T$
\begin{equation*}
\begin{split}
    &\|H^{m}_{\delta_{1}}H^{n}_{\delta_{2}}f(t)\|^{2}_{H_{x}^{2}L^{2}_{v}}\le2A_{0}(A^{m+n-1}(m-2)!(n-2)!)^{2}(1+2C_{6}Te^{2C_{6}T}),
\end{split}
\end{equation*}
plugging it back into \eqref{integration k}, one can deduce
\begin{equation*}
\begin{split}
    &\|H^{m}_{\delta_{1}}H^{n}_{\delta_{2}}f(t)\|^{2}_{H_{x}^{2}L^{2}_{v}}+\int_{0}^{t}|||H^{m}_{\delta_{1}}H^{n}_{\delta_{2}}f(\tau)|||_{H^{2}_{x}, \sigma}^{2}d\tau\\
    &\le 2A_{0}(1+2C_{6}Te^{2C_{6}T})^{2}(A^{m+n-1}(m-2)!(n-2)!)^{2}.
\end{split}
\end{equation*}
We prove then
\begin{equation*}
\begin{split}
    &\|H^{m}_{\delta_{1}}H^{n}_{\delta_{2}}f(t)\|^{2}_{H_{x}^{2}L^{2}_{v}}+\int_{0}^{t}|||H^{m}_{\delta_{1}}H^{n}_{\delta_{2}}f(\tau)|||_{H^{2}_{x}, \sigma}^{2}d\tau\le (A^{m+n-\frac12}(m-2)!(n-2)!)^{2},
\end{split}
\end{equation*}
if we choose the constant $A$ such that
$$2A_{0}(1+2C_{6}Te^{2C_{6}T})^{2}\le A.$$
 We finish the proof of Proposition \ref{prop 4.2}.

\bigskip

{\bf Proof of Theorem  \ref{mainresult}:} Set $\lambda>\frac32$, define $\delta_{1}=\lambda, \delta_{2}=\frac\lambda2+\frac34$. Let $T_{j}, j=1, \cdots, 4$ be the vector fields in section \ref{s2}, then $\partial_{x_1}$ and $A_{+, 1}$ can be generated by the linear combination \eqref{linear combination}. Let $f$ be the smooth solution of Cauchy problem \eqref{1-2} satisfying \eqref{smint}, then for all $\alpha_{1}, m\in\mathbb N$ and $0<t\le T$
\begin{equation}\label{A}
\begin{split}
     &t^{(\lambda+1)\alpha_{1}+\lambda m}\|\partial^{\alpha_{1}}_{x_{1}}A^{m}_{+, 1}f(t)\|_{H^{2}_{x}L^{2}_{v}}\\
     &=\|(T_{1}+T_{2})^{\alpha_{1}}(T_{3}+T_{4})^{m}f(t)\|_{H^{2}_{x}L^{2}_{v}}\\
     &\le\sum_{j=0}^{\alpha_{1}}\sum_{k=0}^{m}C_{\alpha_{1}}^{j}C^{k}_{m}\left|\frac{(\delta_2+ 1)(\delta_1+1)}{\delta_2-\delta_1}\right|^{\alpha_{1}+m}\\
     &\quad\times (T+1)^{(\delta_{1}-\delta_{2})(\alpha_{1}-j+m-k)}\|H^{j+k}_{\delta_{1}}H^{\alpha_{1}+m-j-k}_{\delta_{2}}f(t)\|_{H^{2}_{x}L^{2}_{v}}.
\end{split}
\end{equation}
From Proposition \ref{prop 4.2}, we have that for all $0<t\le T$ and $\alpha_{1}, m\in\mathbb Z_{+}$
\begin{equation*}
\begin{split}
     &\|H^{j+k}_{\delta_{1}}H^{\alpha_{1}+m-j-k}_{\delta_{2}}f(t)\|_{H^{2}_{x}L^{2}_{v}}\\
     &\le A^{\alpha_1+m-\frac12}(j+k-2)!(\alpha_{1}+m-j-k-2)!\\
     &\le A^{\alpha_{1}+m-\frac12}(\alpha_{1}+m)!,
\end{split}
\end{equation*}
with $j=0, 1, \cdots, \alpha_{1}$ and $k=0, 1, \cdots, m$, here we use the fact that $p!q!\le(p+q)!$. Plugging it back into \eqref{A}, since $\delta_{1}>\delta_{2}$ and $A\ge1$, then one can deduce that for all $0<t\le T$
\begin{equation*}
\begin{split}
     &t^{(\lambda+1)\alpha_{1}+\lambda m}\|\partial^{\alpha_{1}}_{x_{1}}A^{m}_{+, 1}f(t)\|_{H^{2}_{x}L^{2}_{v}}\\
     &\le\sum_{j=0}^{\alpha_{1}}\sum_{k=0}^{m}C_{\alpha_{1}}^{j}C^{k}_{m}\left(\frac{(\delta_2+ 1)(\delta_1+1)}{\delta_2-\delta_1}\right)^{\alpha_{1}+m}t^{(\delta_{1}-\delta_{2})(\alpha_{1}+m)}A^{\alpha_{1}+m-\frac12}(\alpha_{1}+m)!\\
     &\le \left(2A(T+1)^{\delta_{1}-\delta_{2}}\frac{(\delta_2+ 1)(\delta_1+1)}{\delta_2-\delta_1}\right)^{\alpha_{1}+m}(\alpha_{1}+m)!.
\end{split}
\end{equation*}
Similarly, the above inequality is also true for $\partial^{\alpha_{1}}_{x_{j}}A^{m}_{+, j}$ with $j=2, 3$.
From \eqref{norm}, we have
\begin{equation*}
\begin{split}
     &\|{\partial_x^{\alpha} \nabla_{\mathcal H_{+}}^{m}  f(t)}\|^{2}_{H^2_{x}L^2_v}=\sum_{|\beta|=m}\frac{m!}{\beta!}\|{\partial_x^{\alpha} A_+^{\beta}  f(t)}\|^2_{H^2_{x}L^2_v}\\
     &\le\sum_{|\beta|=m}\frac{m!}{\beta!}\left(\sum_{j=1}^3\|{\partial^{|\alpha|}_{x_{j}}A^{m}_{+, j}f(t)}\|_{H^2_{x}L^2_v}\right)^2\le 3^{m}\left(\sum_{j=1}^3\|{\partial^{|\alpha|}_{x_{j}}A^{m}_{+, j}f(t)}\|_{H^2_{x}L^2_v}\right)^2,
\end{split}
\end{equation*}
here we use
$$\sum_{|\beta|=m}\frac{m!}{\beta!}=3^{m}, \quad \beta\in\mathbb N^{3}.$$
And therefore,  for any $0<t\le T=1$
\begin{equation*}
\begin{split}
     &t^{(\lambda+1)|\alpha|+ \lambda m} \|{\partial_x^{\alpha} \nabla_{\mathcal H_{+}}^{m}  f(t)}\|_{H^2_{x}L^2_v}
     \le t^{(\lambda+1)|\alpha|+ \lambda m}3^{\frac{m}{2}}\sum_{j=1}^{3}\|\partial^{|\alpha|}_{x_{j}}A^{m}_{+, j}f(t)\|_{H^{2}_{x}L^{2}_{v}}\\
     &\le3\left(6A2^{\delta_1-\delta_2}\frac{(\delta_2+ 1)(\delta_1+1)}{\delta_1-\delta_2} \right)^{|\alpha|+m}(|\alpha|+m)!\leq C^{|\alpha|+m+1} (|\alpha|+m)!,
\end{split}
\end{equation*}
here $C\le\max\{3, 6A2^{\delta_1-\delta_2}\frac{(\delta_2+ 1)(\delta_1+1)}{\delta_1-\delta_2} \}$.

This finishes the proof of main Theorem  \ref{mainresult}.

\section{Proofs of technical Lemmas}\label{s5}
In this section, we give the proof of lemma \ref{lemma5.2} and lemma \ref{lemma5.3}.

\bigskip
\noindent {\bf Proof of Lemma \ref{lemma5.2}.}\\
For all $0<t\le T$, by using Cauchy-Schwarz inequality and \eqref{k=0}, one has
\begin{equation*}
\begin{split}
     &2C_{2}\int_{0}^{t}\|H^{m}_{\delta_{1}}H^{n}_{\delta_{2}}f(\tau)\|_{H^{2}_{x}L^{2}_{v}}|||f(\tau)|||_{H^{2}_{x}, \sigma}|||H^{m}_{\delta_{1}}H^{n}_{\delta_{2}}f(\tau)|||_{H^{2}_{x}, \sigma}d\tau\\
     &\le2C_{2}B\epsilon\|H^{m}_{\delta_{1}}H^{n}_{\delta_{2}}f\|_{L^{\infty}(]0, t]; H^{2}_{x}L^{2}_{v})}\left(\int_{0}^{t}|||H^{m}_{\delta_{1}}H^{n}_{\delta_{2}}f(\tau)|||^{2}_{H^{2}_{x}, \sigma}d\tau\right)^{\frac12}\\
     &\le16(C_{2}B\epsilon)^{2}\|H^{m}_{\delta_{1}}H^{n}_{\delta_{2}}f\|^{2}_{L^{\infty}(]0, t]; H^{2}_{x}L^{2}_{v})}+\frac{1}{16}\int_{0}^{t}|||H^{m}_{\delta_{1}}H^{n}_{\delta_{2}}f(\tau)|||^{2}_{H^{2}_{x}, \sigma}d\tau.
\end{split}
\end{equation*}

We prove now 4 estimates in Lemma \ref{lemma5.3}.

\bigskip
\noindent {\bf Proof of \eqref{R-1+}.}\\
Since $f$ satisfies \eqref{hypothesis k-1} for all $1\le m+n\le k-1$, then for all $0<t\le T$
\begin{equation*}
\begin{split}
     \int_{0}^{t}R_{1}(\tau)d\tau
     &\le\left(1+\frac{1}{C_{1}}\right)\frac{(\delta_{1}m(T+1)^{\delta_{1}-1})^{2}}{8C_{1}}(A^{m+n-\frac32}(m-3)!(n-2)!)^{2}\\
     &\quad+\left(1+\frac{1}{C_{1}}\right)\frac{(\delta_{2}n(T+1)^{\delta_{2}-1})^{2}}{8C_{1}}(A^{m+n-\frac32}(m-2)!(n-3)!)^{2}\\
     &\le A_{1}(A^{m+n-\frac32}(m-2)!(n-2)!)^{2},
\end{split}
\end{equation*}
with $A_{1}=9\left(1+\frac{1}{C_{1}}\right)\frac{(\delta_{1}+\delta_{2})^{2}(T+1)^{2(\delta_{1}+\delta_{2})}}{8C_{1}}$ depends only on $C_{1}, \delta_{1}, \delta_{2}, T$.

\bigskip
\noindent {\bf Proof of \eqref{R-2+}.}\\
Since $f$ satisfies \eqref{hypothesis k-1} for all $1\le m+n\le k-1$, then by using Cauchy-Schwarz, one has for all $0<t\le T$
\begin{equation*}
\begin{split}
     &\int_{0}^{t}R_{2}(\tau)d\tau
     \le\frac{1}{16}\int_{0}^{t}|||H^{m}_{\delta_{1}}H^{n}_{\delta_{2}}f(\tau)|||^{2}_{H^{2}_{x}, \sigma}d\tau\\
     &\quad+12\left(C_{2}A^{m+n-1}\sum_{l=1}^{m}C_{m}^{l}(l-2)!(m-l-2)!\sum_{p=0}^{n-1}C_{n}^{p}(p-2)!(n-p-2)!\right)^{2}\\
     &\quad+12\left(C_{2}A^{m+n-1}(n-2)!\sum_{l=1}^{m-1}C_{m}^{l}(l-2)!(m-l-2)!\right)^{2}\\
     &\quad+12\left(C_{2}A^{m+n-1}(m-2)!\sum_{p=1}^{n}C_{n}^{p}(p-2)!(n-p-2)!\right)^{2}\\
     &\le\frac{1}{16}\int_{0}^{t}|||H^{m}_{\delta_{1}}H^{n}_{\delta_{2}}f(\tau)|||^{2}_{H^{2}_{x}, \sigma}d\tau+A_{2}\left(A^{m+n-1}(m-2)!(n-2)!\right)^{2},
\end{split}
\end{equation*}
here $A_{2}=15\left(24^{2}C_{2}\right)^{2}$ and we use the fact
\begin{equation}\label{summation}
\begin{split}
     \sum_{l=0}^{m}C_{m}^{l}(l-2)!(m-l-2)!\le 24(m-2)!.
\end{split}
\end{equation}

\bigskip
\noindent {\bf Proof of \eqref{R-3+}.}\\
By using Cauchy-Schwarz, one has for all $0<t\le T$
\begin{equation*}
\begin{split}
     &\int_{0}^{t}R_{3}(\tau)d\tau
     \le C_{3}\sum_{l=0}^{m-1}\sum_{p=0}^{n}C_{m}^{l}C^{p}_{n}\left(\sqrt{C_{0}}(T+1)^{\delta_{1}}\right)^{m-l}\left(\sqrt{C_{0}}(T+1)^{\delta_{2}}\right)^{n-p}\\
     &\qquad\times\sqrt{(m-l+n-p)!}\left(\int_{0}^{t}\|H^{l}_{\delta_{1}}H^{p}_{\delta_{2}}f\|^{2}_{H^{2}_{x}L^{2}_{v}}d\tau\right)^{\frac12}\left(\int_{0}^{t}|||H^{m}_{\delta_{1}}H^{n}_{\delta_{2}}f|||^{2}_{H^{2}_{x}, \sigma}d\tau\right)^{\frac12}\\
     &\quad+C_{3}\sum_{p=0}^{n-1}C^{p}_{n}\left(\sqrt{C_{0}}(T+1)^{\delta_{2}}\right)^{n-p}\sqrt{(n-p)!}\left(\int_{0}^{t}\|H^{m}_{\delta_{1}}H^{p}_{\delta_{2}}f\|_{H^{2}_{x}L^{2}_{v}}d\tau\right)^{\frac12}\\
     &\qquad\times\left(\int_{0}^{t}|||H^{m}_{\delta_{1}}H^{n}_{\delta_{2}}f|||_{H^{2}_{x}, \sigma}d\tau\right)^{\frac12}=K_{1}+K_{2}.
\end{split}
\end{equation*}
For the term $K_{1}$, since
\begin{equation}\label{root}
     \sqrt{m!}\le3(m-2)!, \quad \forall \ m\in\mathbb N,
\end{equation}
then by using \eqref{hypothesis k-1} and the fact
$$(p+q)!\le2^{p+q}p!q!, \quad \forall \ p, q\in\mathbb N,$$
we can get that
\begin{equation*}
\begin{split}
     &K_{1}
     \le 9C_{3}\sum_{l=0}^{m-1}\sum_{p=0}^{n}C_{m}^{l}C^{p}_{n}\left(\sqrt{2C_{0}}(T+1)^{\delta_{1}}\right)^{m-l}\left(\sqrt{2C_{0}}(T+1)^{\delta_{2}}\right)^{n-p}\\
     &\ \times(m-l-2)!(n-p-2)!T^{\frac12}A^{l+p-\frac12}(l-2)!(p-2)!\left(\int_{0}^{t}|||H^{m}_{\delta_{1}}H^{n}_{\delta_{2}}f|||^{2}_{H^{2}_{x}, \sigma}d\tau\right)^{\frac12},
\end{split}
\end{equation*}
so that, using \eqref{summation} and taking $A\ge (\sqrt{2C_{0}}(T+1)^{\delta_{1}+\delta_{2}})^{2}+1$, one can get
\begin{equation*}
\begin{split}
     &K_{1}
     \le 9C_{3}T^{\frac12}\sum_{l=0}^{m-1}\sum_{p=0}^{n}C_{m}^{l}(l-2)!(m-l-2)!A^{l+p-\frac12}A^{\frac{m-l+n-p}{2}}\\
     &\qquad\times\sum_{p=0}^{n}C^{p}_{n}(p-2)!(n-p-2)!\left(\int_{0}^{t}|||H^{m}_{\delta_{1}}H^{n}_{\delta_{2}}f|||^{2}_{H^{2}_{x}, \sigma}d\tau\right)^{\frac12},
\end{split}
\end{equation*}
noting that $m-l+n-p\ge1$, then $\frac{m-l+n-p}{2}\le m-l+n-p-\frac12$, from this one can get
\begin{equation*}
\begin{split}
     &K_{1}
     \le 9C_{3}T^{\frac12}A^{m+n-1}\sum_{l=0}^{m-1}\sum_{p=0}^{n}C_{m}^{l}(l-2)!(m-l-2)!\\
     &\qquad\times\sum_{p=0}^{n}C^{p}_{n}(p-2)!(n-p-2)!\left(\int_{0}^{t}|||H^{m}_{\delta_{1}}H^{n}_{\delta_{2}}f|||^{2}_{H^{2}_{x}, \sigma}d\tau\right)^{\frac12}\\
     &\le\frac{1}{64}\int_{0}^{t}|||H^{m}_{\delta_{1}}H^{n}_{\delta_{2}}f|||^{2}_{H^{2}_{x}, \sigma}d\tau+16\left(24^{2}c_{3}A^{m+n-1}(m-2)!(n-2)!\right)^{2},
\end{split}
\end{equation*}
with $c_{3}=9C_{3}\sqrt{T}$.

Similarly, we can obtain that
\begin{equation*}
\begin{split}
     K_{2}
     &\le\frac{1}{64}\int_{0}^{t}|||H^{m}_{\delta_{1}}H^{n}_{\delta_{2}}f|||^{2}_{H^{2}_{x}, \sigma}d\tau+16\left(24c_{3}A^{m+n-1}(m-2)!(n-2)!\right)^{2}.
\end{split}
\end{equation*}
Taking $A_{3}=16\left(24^{2}c_{3}\right)^{2}+16\left(24c_{3}\right)^{2}$, then combine the results of $K_{1}$ and $K_{2}$, one can deduce
\begin{equation*}
\begin{split}
     &\int_{0}^{t}R_{3}(\tau)d\tau\le A_{3}(A^{m+n-1}(m-2)!(n-2)!)^{2}+\frac{1}{32}\int_{0}^{t}|||H^{m}_{\delta_{1}}H^{n}_{\delta_{2}}f(\tau)|||^{2}_{H^{2}_{x}, \sigma}d\tau.
\end{split}
\end{equation*}

\noindent {\bf Proof of \eqref{R-4+}.}\\
Using again Cauchy-Schwarz and \eqref{hypothesis k-1}, one has for all $0<t\le T$
\begin{equation*}
\begin{split}
     &\int_{0}^{t}R_{4}(\tau)d\tau\le C_{4}\sum_{l=0}^{m}\sum_{p=1}^{n}C_{n}^{p}C_{m}^{l}(p+l)^{\frac{[\gamma]}{2}+2}\sqrt{(p+l)!}(T+1)^{\delta_{1} l}(T+1)^{\delta_{2} p}\\
     &\quad\times A^{m-l+n-p-\frac12}(m-l-2)!(n-p-2)!\left(\int_{0}^{t}|||H^{m}_{\delta_{1}}H^{n}_{\delta_{2}}f|||^{2}_{H^{2}_{x}, \sigma}d\tau\right)^{\frac12}\\
     &\quad+C_{4}\sum_{l=0}^{m}C_{m}^{l}l^{\frac{[\gamma]}{2}+1}\sqrt{l!}n(T+1)^{\delta_{1}l}(T+1)^{\delta_{2}}A^{m+n-l-\frac32}(m-l-2)!(n-3)!\\
     &\quad\times\left(\int_{0}^{t}|||H^{m}_{\delta_{1}}H^{n}_{\delta_{2}}f|||^{2}_{H^{2}_{x}, \sigma}d\tau\right)^{\frac12}+C_{4}\sum_{p=1}^{n}C^{p}_{n}p^{\frac{[\gamma]}{2}+2}\sqrt{p!}(T+1)^{\delta_{2} p}A^{m+n-p-\frac12}\\
     &\quad\times (m-2)!(n-p-2)!\left(\int_{0}^{t}|||H^{m}_{\delta_{1}}H^{n}_{\delta_{2}}f|||^{2}_{H^{2}_{x}, \sigma}d\tau\right)^{\frac12}+C_{4}\sum_{p=0}^{n}C_{n}^{p}p^{\frac{[\gamma]}{2}+1}\sqrt{p!}m\\
     &\quad\times (T+1)^{\delta_{1}}(T+1)^{\delta_{2}p}A^{m+n-p-\frac32}(m-3)!(n-p-2)!\left(\int_{0}^{t}|||H^{m}_{\delta_{1}}H^{n}_{\delta_{2}}f|||^{2}_{H^{2}_{x}, \sigma}d\tau\right)^{\frac12}\\
     &=S_{1}+S_{2}+S_{4}+S_{4}.
\end{split}
\end{equation*}
For the term $S_{1}$, using $p^{q}\le e^{p}q!$ and $(p+q)!\le 2^{p+q}p!q!$ for all $p, q\in\mathbb N$, one has
\begin{equation*}
\begin{split}
&S_{1}\le C_{4}\sqrt{([\gamma]+4)!}\sum_{l=0}^{m}\sum_{p=1}^{n}C_{n}^{p}C_{m}^{l}\sqrt{p!l!}(\sqrt{2e}(T+1)^{\delta_{1}})^{l}(\sqrt{2e}(T+1)^{\delta_{2}})^{p}\\
     &\quad\times A^{m-l+n-p-\frac12}(m-l-2)!(n-p-2)!\left(\int_{0}^{t}|||H^{m}_{\delta_{1}}H^{n}_{\delta_{2}}f|||^{2}_{H^{2}_{x}, \sigma}d\tau\right)^{\frac12},
\end{split}
\end{equation*}
taking $A\ge(\sqrt{2e}(T+1)^{\delta_{1}+\delta_{2}})^{2}+1$, one can obtain
\begin{equation*}
\begin{split}
&S_{1}\le C_{4}\sqrt{([\gamma]+4)!}\sum_{l=0}^{m}\sum_{p=1}^{n}C_{n}^{p}C_{m}^{l}\sqrt{p!l!}A^{\frac{l+p}{2}}A^{m-l+n-p-\frac12}\\
     &\quad\times (m-l-2)!(n-p-2)!\left(\int_{0}^{t}|||H^{m}_{\delta_{1}}H^{n}_{\delta_{2}}f|||^{2}_{H^{2}_{x}, \sigma}d\tau\right)^{\frac12},
\end{split}
\end{equation*}
noting that $l+p\ge1$, then $\frac{l+p}{2}\le l+p-\frac12$, by using Cauchy-Schwarz inequality
\begin{equation*}
\begin{split}
&S_{1}\le C_{4}\sqrt{([\gamma]+4)!}\sum_{l=0}^{m}\sum_{p=1}^{n}C_{n}^{p}C_{m}^{l}\sqrt{p!l!}A^{m+n-1}\\
     &\quad\times (m-l-2)!(n-p-2)!\left(\int_{0}^{t}|||H^{m}_{\delta_{1}}H^{n}_{\delta_{2}}f|||^{2}_{H^{2}_{x}, \sigma}d\tau\right)^{\frac12}\\
     &\le\frac{1}{128}\int_{0}^{t}|||H^{m}_{\delta_{1}}H^{n}_{\delta_{2}}f|||^{2}_{H^{2}_{x}, \sigma}d\tau\\
&\quad+32\left(c_{4}A^{m+n-1}\sum_{l=0}^{m}\sum_{p=1}^{n}C_{n}^{p}C_{m}^{l}\sqrt{p!l!}(m-l-2)!(n-p-2)!\right)^{2},
\end{split}
\end{equation*}
here $c_{4}=C_{4}\sqrt{([\gamma]+4)!}$. So that, from \eqref{summation} and \eqref{root}, we can conclude
\begin{equation*}
\begin{split}
S_{1}
&\le\frac{1}{128}\int_{0}^{t}|||H^{m}_{\delta_{1}}H^{n}_{\delta_{2}}f|||^{2}_{H^{2}_{x}, \sigma}d\tau+32\left(96^{2}c_{4}A^{m+n-1}(m-2)!(n-2)!\right)^{2}.
\end{split}
\end{equation*}
Similarly, we can get that for $j=2, 3, 4$
\begin{equation*}
\begin{split}
S_{j}
&\le\frac{1}{128}\int_{0}^{t}|||H^{m}_{\delta_{1}}H^{n}_{\delta_{2}}f|||^{2}_{H^{2}_{x}, \sigma}d\tau+32\left(96c_{4}A^{m+n-1}(m-2)!(n-2)!\right)^{2}.
\end{split}
\end{equation*}
And therefore, combining the results of $S_{1}-S_{4}$, then taking $A_{4}=35\left(96^{2}c_{4}\right)^{2}$, we can get that
\begin{equation*}
\begin{split}
     &\int_{0}^{t}R_{4}(\tau)d\tau\le A_{4}(A^{m+n-1}(m-2)!(n-2)!)^{2}+\frac{1}{32}\int_{0}^{t}|||H^{m}_{\delta_{1}}H^{n}_{\delta_{2}}f(\tau)|||^{2}_{H^{2}_{x}, \sigma}d\tau.
\end{split}
\end{equation*}

\bigskip
\bigskip
\noindent {\bf Acknowledgements.}
This work was supported by the NSFC (No.12031006) and the Fundamental
Research Funds for the Central Universities of China.

\end{document}